\documentclass[11pt]{article}
\usepackage{graphicx}
\usepackage{bm}
\usepackage{amsmath,amsthm,amssymb}
\usepackage{multirow,bigdelim}
\usepackage{amscd}
\usepackage{multirow}

\newtheorem{thm}{Theorem}[section]
\newtheorem{theorem}[thm]{Theorem}
\newtheorem{proposition}[thm]{Proposition}
\newtheorem{lemma}[thm]{Lemma}

\newtheorem{corollary}[thm]{Corollary}

\DeclareMathOperator{\rank}{rank}
\DeclareMathOperator{\corank}{corank}

\DeclareMathOperator{\spa}{span}

{%
   \begin{list}{$\bullet$\ \ }
   {%
      \setlength{\itemindent}{0pt}
      \setlength{\leftmargin}{3zw}
       \setlength{\rightmargin}{0zw}
       \setlength{\labelsep}{0zw}
       \setlength{\labelwidth}{3zw}
       \setlength{\itemsep}{0em}
       \setlength{\parsep}{0em}
       \setlength{\listparindent}{0zw}
   }
}{%
   \end{list}%
}

\begin{document}
\title{Singularity Degree of the Positive Semidefinite Matrix Completion Problem}
\author{Shin-ichi Tanigawa\thanks{
Centrum Wiskunde $\&$ Informatica (CWI), Postbus 94079, 1090 GB
Amsterdam, The Netherlands, and
Research Institute for Mathematical Sciences,
Kyoto University, Sakyo-ku, Kyoto 606-8502, Japan.
e-mail: {\tt tanigawa@kurims.kyoto-u.ac.jp}}}
\maketitle

\begin{abstract}
The singularity degree of a semidefinite programming problem is the smallest number of facial reduction steps to make the problem strictly feasible.
We introduce two new graph parameters, called the singularity degree
and the nondegenerate singularity degree, based on the singularity degree of the positive semidefinite matrix completion problem.
We give a characterization of the class of graphs whose parameter value is at most one for each parameter.
Specifically, we show that the singularity degree of a graph is at most one if
and only if the graph  is chordal, and the nondegenerate singularity
degree of a graph is at most one if and only if the graph  is the clique sum of chordal graphs and $K_4$-minor free graphs.
We also show that the singularity degree is bounded by two if the treewidth is bounded by two,
and exhibit  a family of graphs with treewidth  three, whose singularity degree grows linearly in the number of vertices.
%


\end{abstract}

\section{Introduction}
In the {\em positive semidefinite (PSD) matrix completion problem}, 
we are given a {\em partial matrix}, i.e.,  a matrix whose entries are specified only on a subset of the positions,  and we are asked to fill the missing entries so that the resulting matrix is positive semidefinite.
More formally, given an undirected graph $G=(V,E)$ with vertex set $V=\{1,\dots, n\}$ and an edge weight $c:E\rightarrow [-1,1]$,   the (real) PSD matrix completion problem ${\rm P}(G,c)$ asks to decide whether the following set is empty, and if not  to find a point in it\footnote{Having all diagonal entries specified is a relatively common assumption in this context. 
By allowing $G$ to have self-loops, it is possible to describe a problem where some of diagonal entries are missing. However due to \cite[Lemma 4.2]{dpw}  the problem of analyzing the singularity degree can be always reduced to the subgraph induced by the set of looped vertices. Hence we can safely assume that all the diagonal entries are specified.}:
\[
\left\{X\in{\cal S}^n_+\mid X[i,i]=1 \ \forall i\in V,  \  X[i,j]=c(ij) \ \forall ij\in E \right\}
\]
where ${\cal S}^n_+$ denotes the set of positive semidefinite matrices of size $n$.
Let ${\cal E}_n:=\{X\in {\cal S}^n_+\mid X[i,i]=1 \ \forall i\in V\}$, which is known as the {\em elliptope}.
The  set of edge weights $c$ for which the program ${\rm P}(G,c)$ is feasible is the projection of the elliptope along the coordinate axes~\cite{l97,l98}, 
called a {\em coordinate shadow of the elliptope} and 
denoted by ${\cal E}(G)$.

The {\em singularity degree} of convex cone programming introduced by Sturm~\cite{s00} is defined as  the smallest number of facial reduction steps (in the sense of  Borwein and Wolkowicz~\cite{bw})  to make the problem strictly feasible.
In this paper we shall analyze relations between the singularity degree and the underlying combinatorics by focusing on the PSD matrix completion problem.
%
%
Based on the singularity degree of the PSD matrix completion problem, we introduce two  new graph parameters.
The first parameter is  the {\em singularity degree} ${\sf sd}(G)$ of a graph $G$, which is defined as the largest singularity degree of the matrix completion problem ${\rm P}(G,c)$ over all edge weights $c\in {\cal E}(G)$.
The second parameter is defined in terms of the singularity degree over all nondegenerate inputs.
We say that $c:E\rightarrow  [-1,1]$ is {\em nondegenerate} if $c(ij)\notin \{-1,1\}$ for every edge $ij\in E$
and define  the {\em nondegenerate singularity degree} ${\sf sd}^*(G)$ to be the largest singularity degree of ${\rm P}(G,c)$ over all nondegenerate edge weights $c\in {\cal E}(G)$.
The elimination of degenerate inputs is motivated from both practical and theoretical view points.\footnote{
If $c(ij)=1$, then $X[i,j]=1$ and $X[i,k]=X[j,k]$ for any feasible $X$ and any $k\in V(G)\setminus \{i,j\}$. This implies that the feasible set of ${\rm P}(G,c)$ is equal to that of ${\rm P}(G/ij, c)$, where $G/ij$ is the graph obtained from $G$ by contracting $ij$, and we can always  focus on the contracted smaller problem. (In the Euclidean matrix completion problem, such a degeneracy corresponds to the case when  the distance between $i$ and $j$ is specified to be  zero, in that case $i$ and $j$ being recognized as just one point.) A similar trivial reduction is also possible if $c(ij)=-1$.}

We investigate classes of graphs having small (nondegenerate) singularity degree. Our main results are summarized as follows:
\begin{itemize}
\setlength{\parskip}{0.cm} 
 \setlength{\itemsep}{0.05cm} 
\item ${\sf sd}(G)\leq 1$ if and only if $G$ is chordal (Theorem~\ref{thm:sd1}).
\item ${\sf sd}^*(G)\leq 1$ if and only if $G$ is the clique sum of chordal graphs and $K_4$-minor free graphs (Theorem~\ref{thm:nonsingular}).
\item ${\sf sd}(G)\leq 2$ if $G$ is the clique sum of chordal graphs and $K_4$-minor free graphs (Corollary~\ref{thm:singular_two}).
\item For any positive integer $n$, there is a graph of $n$ vertices  with treewidth three, whose singularity degree is at least $\lfloor (n-1)/3\rfloor$ (Corollary~\ref{cor:large}).
\end{itemize}
We also show that ${\rm sd}(G)=0$ if and only if $E(G)=\emptyset$, and ${\rm sd}^*(G)=0$ if and only if $G$ is $K_3$-minor free.

Our work is motivated by  a recent paper by Drusvyatskiy, Pataki, and Wolkowicz~\cite{dpw}, who analyzed the facial structures of ${\cal E}(G)$ by using the facial reduction algorithm under the general setting of looped graphs.
Specifically, it was shown that 
the singularity degree  is  at most one if and only if the face of $c$ in ${\cal E}(G)$ is {\em exposed} (i.e., it can be written as the intersection of ${\cal E}(G)$ and a supporting hyperplane).
One of their main theorems implies that ${\sf sd}(G)\leq 1$ if $G$ is chordal, and they also posed a question about characterizing graphs with ${\sf sd}(G)=1$.
The first our main theorem gives an answer to this question (see also a discussion after Theorem~\ref{thm:sd1}).

One of  difficult cases for proving the main theorems is to bound the singularity degree of $K_4$-minor free graphs. 
To analyze this class, our idea is to look at the singed version of the PSD matrix completion problem and prove a generalized statement for signed graphs. 
Since the proof is technical and involved, this case is handled in a separate paper~\cite{t}.

It is known that the facial reduction eventually produces a certificate for the rank maximality of a primal feasible solution in SDP.
As shown by Connelly and Gortler~\cite{cg}, this certificate can be used for characterizing uniquely solvable SDP problems. The concept of unique solvability in the PSD matrix completion problem coincides with the {\em universal rigidity} of {\em spherical (bar-joint) frameworks} in rigidity theory.
Adapting the analysis for bounding the singularity degree, we also give a new characterization of graphs whose universal rigidity can be characterized by Connelly's super stability condition.

%


The paper is organized as follows.
Section~\ref{sec:facial} and Section~\ref{sec:psd}  aim to survey the latest connections between semidefinite programming and graph rigidity.
In Section~\ref{sec:facial} we  review the facial reduction and give a formal definition of the singularity degree for general SDP. We also discuss about how it is related to the unique solvability of SDP.
In Section~\ref{sec:psd} we move to the PSD matrix completion problem.
By specializing the facial reduction to the PSD matrix completion problem, we define the singularity degree of graphs. 
In Section~\ref{sec:degree_one} we prepare several fundamental facts to analyze the singularity degree,
and we give a characterization of graphs whose singularity degree is at most one.
In Section~\ref{sec:nondegenerate} we give a characterization of graphs whose nondegenerate singularity degree is at most one,
and in Section~\ref{subsec:two} we discuss about a family of graphs whose singularity degree is bounded by two.
On the other hand, in Section~\ref{sec:large} we give a family of graphs whose singularity degree cannot be bounded by sublinear in the number of vertices.
In Section~\ref{sec:super}, using the analysis of bounding the singularity degree, we give a characterization of the universal rigidity of spherical frameworks.

\section{Facial Reduction, Singularity Degree, Unique Solvability}
\label{sec:facial}
In this section we summarize the facial reduction process, which is required for the formal definition of the singularity degree. 
We also give a comment on the unique solvability of SDP. 
This is an interpretation of the recent result by Connelly and Gortler \cite{cg} for universal rigidity (discussed later) in terms of general SDP, combined with a discussion on degeneracy given in \cite{lv}.
Although the presentation is different, all the key ideas in this section are already given in \cite{a15,lp,cg,lv}.

Throughout the paper the set of symmetric matrices of size $n\times n$ is denoted by ${\cal S}^n$.
We say that a symmetric matrix $S\in {\cal S}^n$ is {\em positive semidefinite on a linear space $W\subseteq \mathbb{R}^n$} if $x^{\top} Sx\geq 0$ for every $x\in W$.
If $W=\mathbb{R}^n$, then $S$ is simply called a positive semidefinite matrix, and  the set of positive semidefinite matrices of size $n\times n$ is denoted by ${\cal S}^n_+$.

\subsection{Facial Reduction}
\label{subsec:facial}
Let $n$ and $m$ be nonnegative  integers and  $E=\{1,\dots, m\}$.
We consider the following SDP feasibility problem and its dual:
\[
\begin{array}{clcc}
{\rm (P)} & \sup & 0  &\\
&\text{s.t.} & \langle A_i, X\rangle \leq b_i & (i\in E) \\
&& X\in {\cal S}^n_+ & 
\end{array}
\quad 
\begin{array}{clc}
{\rm (D)} & \inf & \sum_{i\in E} b_iy_i  \\
&\text{s.t.} & \sum_{i\in E} y_i A_i \succeq 0  \\
& & y_i\geq 0 \\
& & y\in \mathbb{R}^E.
\end{array}
\]
By a theorem of alternatives for SDP, either there exists a positive definite feasible solution $X$ in (P) or there exists a dual feasible solution $y\in \mathbb{R}^E$ with $Y:=\sum_{i\in E} y_i A_i\neq 0$ and $\sum_{i\in E} b_iy_i\leq 0$.  
The latter case can be split into two subcases depending on whether $\sum_{i\in E} b_iy_i< 0$ or not.
If $\sum_{i\in E} b_iy_i< 0$, $y$ certificates that (P) is infeasible.
On the other hand, if $\sum_{i\in E} b_iy_i=0$, then $\langle X, Y\rangle= 0$ holds for any feasible solution $X$ of (P), and hence we may insert $\langle X, Y\rangle=0$ as a new constraint into (P) to get a stronger dual. 

In general, suppose we found nonzero  $Y^1,\dots, Y^{k}$ such that 
(i) $\langle X, Y^j\rangle=0$ for all feasible $X$ of (P) and 
(ii) $Y^j$ is positive semidefinite on ${\cal V}^{j-1}$ 
for each $j$ with $1\leq j\leq k$,
where 
\begin{equation}
\label{eq:calV}
{\cal V}^{j}:=\{x\in \mathbb{R}^n\mid x^{\top}Y^i x=0 \ (i=1,\dots, j)\},
\end{equation}
for $1\leq i\leq k$ and ${\cal V}^0:=\mathbb{R}^n$.
Then we may restrict the cone ${\cal S}^n_+$ in (P) to 
\begin{equation}
\label{eq:S_k}
{\cal S}_{+,k}^n:={\rm cone}\{xx^{\top}\mid x\in {\cal V}^{k}\}
\end{equation}
since we have 
${\cal S}_{+,k}^n=\{X\in {\cal S}^n_+\mid \langle X, Y^j\rangle=0\ (j=1,\dots,k)\}$. 
Note that ${\cal S}_{+,k}^n$ is linearly isomorphic to ${\cal S}_+^r$ with $r=\dim {\cal V}^{k}$,
and the dual cone $({\cal S}_{+,k}^n)^{*}$ consists of symmetric matrices which are positive semidefinite on ${\cal V}^{k}$.
We then consider 
\[
\begin{array}{clcc}
{\rm (P^k)} & \sup & 0  &\\
&\text{s.t.} & \langle A_i, X\rangle \leq b_i & (i\in E) \\
&& X\in {\cal S}_{+,k}^n & 
\end{array}
\quad 
\begin{array}{clc}
{\rm (D^k)} & \inf & \sum_{i\in E} b_iy_i  \\
&\text{s.t.} & \sum_{i\in E} y_i A_i\in  ({\cal S}_{+,k}^n )^* \\
& & y_i\geq 0.
\end{array}
\]
Since ${\cal S}_{+,k}^n$ is linearly isomorphic to a PSD cone, a theorem of alternatives again implies that either there exists a feasible solution $X$ of (P) that is positive definite on ${\cal V}^k$ or 
there exists a feasible solution $y^{k+1}\in \mathbb{R}^E$ of ${\rm (D^k)}$ such that $Y^{k+1}:=\sum_{i\in E} y_i^{k+1} A_i$ is nonzero on ${\cal V}^k$ (i.e., $\exists x\in {\cal V}^k$ with $x^{\top} Y^{k+1} x> 0$) and $\sum_{i\in E} b_iy_i^{k+1}\leq 0$.  
In the latter case, we have $Y^{k+1}\in  ({\cal S}^n_{+,k})^*$ and $X\in {\cal S}^n_{+,k}$ for any feasible solution $X$ of (P), and hence 
\[
0\leq \langle X, Y^{k+1}\rangle=\sum_{i\in E} \langle X, y_i^{k+1}A_i\rangle\leq \sum_{i\in E} b_iy_i^{k+1}.
\] 
Thus, if  $\sum_{i\in E} b_iy_i^{k+1}< 0$ then the sequence $\{y^1, \dots, y^{k+1}\}$ certifies the infeasibility of (P).
Otherwise,  $\langle Y^{k+1}, X\rangle=0$ holds for every feasible solution $X$ of (P).
Since $Y^{k+1}$ is positive semidefinite on ${\cal V}^{k}$, we can repeat the process by inserting $\langle X, Y^{k+1}\rangle=0$ as a new constraint into ${\rm (P^k)}$.  
Moreover, since $Y^{k+1}$ is nonzero on ${\cal V}^k$,   ${\cal V}^{k+1}$ will be strictly smaller than ${\cal V}^{k}$, and the process will be terminated with at most $n$ iterations, where at the last iteration we obtain either an interior point of the feasible set of (P) or a certificate for the infeasibility of (P).

Motivated by terminologies introduced by Connelly and Gortler~\cite{cg} in rigidity context, 
we say that a sequence $\{Y^1, \dots, Y^k\}$ of symmetric matrices  
is {\em nested PSD of rank $d$} if 
$Y^i$ is positive semidefinite on ${\cal V}^{i-1}$ with $\dim {\cal V}^k=n-d$.
Summarizing the above discussion,  the sequence $\{y^1,\dots, y^k\}$ of vectors obtained in the facial reduction satisfies the following properties:
\begin{description}
\setlength{\parskip}{0.cm} 
 \setlength{\itemsep}{0.05cm} 
\item[(c1)] $y^j\geq 0$ for every $j\in [k]$.
\item[(c2)] $\{Y^1,\dots, Y^k\}$ is nested PSD (where $Y^i=\sum_{i\in E} y_i^k A_i$).
\item[(c3)] $\sum_{i\in E} y^j_i b_i\leq 0$ for each $j\in [k]$.
\end{description}
The sequence also enjoys the following relation with the feasible set of (P):
\begin{description}
\setlength{\parskip}{0.cm} 
 \setlength{\itemsep}{0.05cm} 
\item[(c4)] The rank of $\{Y^1,\dots, Y^k\}$ is  $n-s$, where $s$ is the rank of an interior point of the feasible set of (P).
\end{description}
The smallest size of a sequence $\{y^1,\dots, y^{k}\}$ satisfying (c1)(c2)(c3)(c4) is called the {\em singularity degree} of (P).
If (P) has a strict feasible solution (i.e., a positive definite feasible solution), then the singularity degree is set to be zero.

We give a list of further properties,  which are implicit in the above discussion and whose easy proofs are omitted.
\begin{proposition}
\label{prop:sequence_prop}
Suppose that a sequence $\{y^1,\dots, y^k\}$ of vectors in $\mathbb{R}^E$ 
satisfies (c1), (c2), and (c3) with rank $r$.
Then the following hold.
\begin{itemize}
\setlength{\parskip}{0.cm} 
 \setlength{\itemsep}{0.05cm} 
\item $x^{\top}Y^j x=0$ for any eigenvector $x$ of a feasible solution $X$ of (P).
\item If $y_i^j>0$ for some $j$, then $\langle A_i, X\rangle=b_i$ for any feasible $X$ of (P).
\item $\rank X\leq n-r$ for any feasible $X$ of (P).
\end{itemize}
\end{proposition}
%
%
%
The facial reduction process and Proposition~\ref{prop:sequence_prop} imply the following.

\begin{proposition}
\label{prop:rank}
A feasible solution $X$ of (P) attains the maximum rank 
if and only if  there is a sequence $\{y^1,\dots, y^k\}$ of vectors in $\mathbb{R}^E$ satisfying (c1), (c2), and (c3) with rank $r=n-\rank X$.
\end{proposition}
%

A pair $(X, y)$ of a primal and a dual feasible solutions are said to satisfy the {\em complementarity condition} if $\langle X, Y\rangle= 0$ and $(\langle A_i, X\rangle - b_i) y_i = 0$ for every $i\in E$, where $Y=\sum_{i\in E} y_iA_i$.  The complementarity condition implies $\rank X + \rank Y \leq n$.
If $\rank X + \rank Y =n$,  the pair is said to satisfy the {\em strict} complementarity condition. 
%
By Proposition~\ref{prop:sequence_prop} and Proposition~\ref{prop:rank}, in a feasible problem, 
the strict complementarity condition holds  if and only if the singularity degree is at most one.

 \subsection{Unique solvability of SDP}
 \label{subsec:unique}
 We give a supplementary comment on the uniqueness of solutions.
 Let $J_{\{y^1,\dots, y^k\}}=\{i\in E: y_i^j>0 \text{ for some $j$\}}$
and let $J_P=\{i\in E: \langle A_i, X\rangle=b_i \text{ for every feasible $X$ of (P)} \}$. 
If $\{y^1,\dots, y^k\}$ is obtained by the facial reduction, Proposition~\ref{prop:sequence_prop} implies $J_{\{y^1,\dots, y^k\}}\subseteq J_P$. This inclusion may be strict, but one can enforce the equality by adding new vectors into the list.
To see this we consider
\[
\begin{array}{clcc}
 & \sup & t  &\\
&\text{s.t.} & \langle A_i, X\rangle \leq b_i & (i\in J_{\{y^1,\dots, y^k\}}) \\ 
& & \langle A_i, X\rangle \leq b_i-t & (i\notin J_{\{y^1,\dots, y^k\}}) \\ 
&& X\in {\cal S}^n_{+,k}  
\end{array}
\quad 
\begin{array}{clc}
& \inf & 0 \\
&\text{s.t.} & \sum_{i\in E} y_iA_i\in ({\cal S}^n_{+,k})^*  \\
& & \sum_{i\in E} y_i b_i\leq 0 \\
& & \sum_{i\notin J_{\{y^1,\dots, y^k\}}} y_i=1 \\
& & y_i\geq 0, z_i\geq 0.
\end{array}
\]
Note that the primal is strictly feasible after the facial reduction, and hence a dual feasible solution $y$ would certificate a new tight inequality (i.e., some $i\notin J_{\{y^1,\dots, y^k\}}$ such that $\langle A_i, X\rangle=b_i$ for every feasible $X$). Then we can insert the new $y$ into the list keeping the properties (c1)(c2)(c3)(c4). 
One can repeat this until the dual becomes infeasible. If the dual becomes infeasible, the primal is unbounded by the strict duality (as the primal is strict feasible), and we obtain 
\begin{description}
\item[(c5)] $y_i^j>0$ for some $j$ if and only if $\langle A_i, X\rangle=b_i$ for any feasible  $X$ of (P).
\end{description}

The following characterization of the uniqueness of the solution is an interpretation of the result by Connelly and Gortler in \cite{cg} in terms of general SDP.
 \begin{proposition}
 \label{prop:unique}
 Suppose that (P) is feasible.
 Then (P) has a unique solution whose rank is equal to $s$ if and only if there is a sequence $\{y^1, \dots, y^k\}$ of vectors in $\mathbb{R}^E$
 satisfying (c1), (c2), (c3), (c5) with rank $n-s$, and the following nondegeneracy condition:
 \begin{description}
 \item[(c6)] ${\rm span}\{xx^{\top}\mid x\in {\cal V}^k\} \cap ({\rm span}\{A_i: i\in J_{ \{ y^1, \dots, y^k \}} \})^{\top} =\{0\}$,
 \end{description}
 where ${\cal V}^k=\{x\in \mathbb{R}^n: x^{\top}Y^j x=0\ (j=1,\dots, k)\}$ as defined in (\ref{eq:calV}).
 \end{proposition}
 \begin{proof}
Suppose (P) has a unique solution  $X$ of rank $s$.
Applying the facial reduction one can always obtain a sequence $\{y^1, \dots, y^k\}$ 
 satisfying (c1)(c2)(c3)(c5) with rank $n-s$.
If (c6) does not hold, then there is nonzero $S\in {\cal S}^n$ 
such that $\langle S, A_i\rangle=0$ for $i\in J_{\{y^1,\dots, y^k\}}$ and 
each eigenvector of $S$ is contained in ${\cal V}^k$.
Since $X$ is positive definite on ${\cal V}^k$, 
$X+\epsilon S\succeq 0$ for any sufficiently small $\epsilon>0$.
Moreover, for any sufficiently small $\epsilon>0$, $\langle X+\epsilon S, A_i\rangle\leq b_i$ follows from (c5).
Thus $X+\epsilon S$ is feasible, contradicting the uniqueness.
 
 Conversely, suppose that there is a sequence satisfying the condition of the statement.
 Take any two interior points $X_1$ and $X_2$ of the feasible region of (P).
 By $X_i\in {\cal S}^n_{+,k}$, 
 we have $X_1-X_2\in  {\rm span}\{ xx^{\top} \mid x\in {\cal V}^k\}$.
 On the other hand, by Proposition~\ref{prop:sequence_prop}(ii), 
 $\langle A_i, X_j\rangle=b_i$ for $i\in J_{\{y^1,\dots, y^k\}}$ and $j=1,2$.
 Thus  we have $X_1-X_2\in ({\rm span}\{A_i: i\in J_{ \{ y^1, \dots, y^k \}} \})^{\top}$. 
 Therefore, by (c6), we get $X_1=X_2$.
 \end{proof}
 
 Condition (c6)  given in the statement is motivated from the nondegeneracy condition in strictly feasible SDP  problems given in \cite{ahm}.
 (Note that, if $Y\succeq 0$, then 
 $S\in {\rm span}\{xx^{\top} \mid x^{\top} Y x=0\}$ if and only if $SY=0$ for any symmetric matrix $S$.) 
 In the PSD completion problem, this form corresponds to the so-called Strong Arnold property  (see, e.g.,~\cite{lv}).
 
 It was pointed out in \cite{lv} that the nondegeneracy condition can be described as a relation with a primal maximum rank solution. In the same manner, (c6) can be written as follows:
 \begin{description}
 \item[(c7)] There is no nonzero symmetric matrix $S\in {\cal S}^d$ satisfying 
 $\langle P^{\top}S P, A_i\rangle=0$ for every $i\in J_{y^1,\dots, y^k}$,
 where $d=\dim {\cal V}^k$ and $P$ is a $d\times n$ matrix such that 
 $P^{\top}P$ is an interior point of  ${\cal S}^n_{+,k}={\rm cone}\{xx^{\top}\mid x\in {\cal V}^k\}$.
 \end{description}
 Indeed, any symmetric matrix $S\in {\cal S}^d$ can be written as $N^{\top}N-M^{\top}M$ for some  $d\times d$ matrices $N$ and  $M$.
 Hence $P^{\top}SP$ can be written as $(NP)^{\top}(NP)-(MP)^{\top}(MP)$.
 Thus, $P^{\top}SP$ is another way of expressing matrices in ${\rm span}\{xx^{\top} \mid x\in {\cal V}^k\}$.
%

\section{PSD Matrix Completion Problem}
In this section we define the singularity degree of graphs by specializing the facial reduction process to the PSD matrix completion problem. 
In Section~\ref{subsec:singularity_rigidity} we  discuss how  the new parameter is related to existing notion from rigidity.
In Section~\ref{subsec:singularity_example} we give two examples of computations of singularity degree.

We use the following notation throughout the paper.  
For an undirected graph $G$, $V(G)$ and $E(G)$ denote the vertex and the edge set of $G$, respectively.
If $G$ is clear from the context, we simply use $V$ and $E$ to denote $V(G)$ and $E(G)$, respectively.
For $X\subseteq V$, let $\delta(X)$ denotes the set of edges between $X$ and $V\setminus X$.
If $X=\{v\}$ for some $v\in V$, then $\delta(\{v\})$ is simply denoted by $\delta(v)$. 

Let $G$ be a graph. A pair of subgraphs $\{G_1, G_2\}$ is called a {\em cover} of $G$ if
$V(G_1)\cup V(G_2)=V(G)$ and $E(G_1)\cup E(G_2)=E(G)$.
We say that $G$ is a {\em clique sum} of $G_1$ and $G_2$ if 
$\{G_1, G_2\}$ is a cover of $G$ and 
$V(G_1)\cap V(G_2)$ induces a clique in $G$.

For a finite set $X$,   let $\mathbb{R}^X$ be a $|X|$-dimensional vector space each of whose coordinate is associated with an element in $X$.

We denote ${\cal E}_n=\{X\in {\cal S}_+^n: \forall i, X[i,i]=1 \}$.
The projection of ${\cal E}_n$ to $\mathbb{R}^E$ is denoted by 
$\pi_G: {\cal E}_n\rightarrow \mathbb{R}^E$, i.e., $(\pi_G(X))_{ij}=X[i,j]$ for $ij\in E$.
Let ${\cal E}(G)=\pi_G({\cal E}_n)$ as given in the introduction.
Also let ${\bm e}_i$ be the $n$-dimensional vector whose $i$-th coordinate is one and the other entries are zero.

\label{sec:psd}
\subsection{SDP formulation}
Given a graph $G$ and $c\in [-1,1]^{E}$, the PSD matrix completion problem, 
denoted by ${\rm P}(G,c)$, and its dual are  formulated as follows:
\begin{equation*}
\begin{array}{lll}
 \sup & 0 & \\
{\rm s.t.} & X[i,j]= c(ij) & (ij\in E) \\
 & X[i,i]=1 & (i\in V) \\
 & X\succeq 0
\end{array}
\quad
\begin{array}{lll}
 \inf & \sum_{i\in V} \omega(i)+ \sum_{ij\in E} \omega(ij)c(ij)  & \\
{\rm s.t.}  & \sum_{i\in V} \omega(i) E_{ii}+ \sum_{ij\in E} \omega(ij)E_{ij}\succeq 0 \\
 & \omega\in \mathbb{R}^{V\cup E}
\end{array}
\end{equation*}
where $E_{ij}=({\bf e}_i{\bf e}_j^{\top}+{\bf e}_j{\bf e}_i^{\top})/2$.
Throughout the paper, for $\omega\in \mathbb{R}^{V\cup E}$, we  use the capital letter $\Omega$ to denote $\sum_{i\in V} \omega(i) E_{ii}+ \sum_{ij\in E} \omega(ij)E_{ij}$.

The singularity degree of ${\rm P}(G,c)$ is denoted by ${\sf sd}(G,c)$.
Based on the worst singularity degree, one can define the {\em singularity degree} of $G$ as follows:
\[
{\sf sd}(G)=\max \{{\sf sd}(G,c) \mid c\in {\cal E}(G) \}.
\]
In order to analyze the singularity degree it is sometimes required to impose sings of entries in dual solutions. For this purpose it is reasonable to introduce the inequality setting of the completion problem.
A signed graph $(G,\Sigma)$ is a pair of a graph $G=(V,E)$ and $\Sigma\subseteq E$,
where $G$ may contain parallel edges.
We sometime abuse notation $ij$ to denote an edge between $i$ and $j$ if the sign is clear from the context.

Given a signed graph $(G,\Sigma)$ and $c\in [-1, 1]^{E}$, we are interested in the following SDP, denoted by ${\rm P}(G,\Sigma,c)$, and its dual:
\begin{equation*}
\begin{array}{lll}
\sup & 0 & \\
{\rm s.t.} & X[i,j]\geq c(ij) & (ij\in E\setminus \Sigma) \\
 & X[i,j]\leq c(ij) & (ij\in E\cap \Sigma) \\
 & X[i,i]=1 & (i\in V) \\
 & X\succeq 0
\end{array}
\quad
\begin{array}{lll}
 \inf & \sum_{i\in V} \omega(i)+ \sum_{ij\in E} \omega(ij)c(ij)  & \\
{\rm s.t.} & \omega(ij)\leq 0 & (ij\in E\setminus \Sigma) \\
 & \omega(ij)\geq 0 & (ij\in E\cap \Sigma) \\
 & \sum_{i\in V} \omega(i) E_{ii}+ \sum_{ij\in E} \omega(ij)E_{ij}\succeq 0 \\
 & \omega\in \mathbb{R}^{V\cup E}
\end{array}
\end{equation*}
The signed version of ${\cal E}(G)$ can be defined as 
\begin{equation*}
{\cal E}(G,\Sigma)=\{c\in [-1,1]^E \mid {\rm P}(G, \Sigma, c) \text{ is feasible} \}.
\end{equation*}
As in the undirected case, the singularity degree of signed graphs can be defined by 
\[
{\sf sd}(G,\Sigma)=\max\{{\sf sd}(G,\Sigma,c)\mid c\in {\cal E}(G,\Sigma)\}.
\]
where ${\sf sd}(G,\Sigma,c)$ denotes the singularity degree of ${\rm P}(G,\Sigma, c)$.
%

\subsection{Connection to spherical rigidity and unique completability}
\label{subsec:singularity_rigidity}
As discussed in Section~\ref{subsec:unique}, the singularity degree is closely related to the unique solvability of  SDP. In the PSD matrix completion problem, checking the uniqueness of a  completion corresponds to a rigidity question about the realization of the underlying graph on the unit sphere. 
Thus several concepts from rigidity theory will be useful to analyze the singularity degree.
Terms defined in this subsection are motivated from the corresponding properties in rigidity theory or the graph realization problem~\cite{c,cg,gt,zsy}.

A {\em spherical framework} is a pair $(G, p)$ of a graph and $p:V\rightarrow \mathbb{S}^d$ for some $d$.
More generally, a pair $(G, \Sigma, p)$ of a signed graph $(G,\Sigma)$ and $p:V\rightarrow \mathbb{S}^d$ is called 
a {\em spherical tensegrity}.
A positive semidefinite matrix $X$ of rank $d$ can be represented as $P^{\top}P$ for some $d\times n$ matrix of rank $d$. 
This representation is referred to as a {\em Gram matrix representation} of $X$.
By assigning the $i$-th column of $P$ with each vertex, one can obtain a map $p:V\rightarrow \mathbb{R}^d$. 
If $X\in {\cal E}_n$ (i.e.,  $X[i,i]=1$ for every $i$), $p$ is actually a map to the unit sphere $\mathbb{S}^{d-1}$,
and hence $(G,\Sigma,p)$ forms a spherical tensegrity.
Conversely, any $p:V\rightarrow \mathbb{S}^{d-1}$ defines $X\in {\cal E}_n$ of rank $d$ 
by $X[i,j]=p(i)\cdot p(j)$. This $X$ is denoted by ${\rm Gram}(p)$.


In the context of graph rigidity, a vector $\omega\in \mathbb{R}^{V\cup E}$ is said to be a {\em spherical stress} or (simply {\em stress}).
A stress $\omega\in \mathbb{R}^{V\cup E}$ is said to be {\em equilibrium}  if the following {\em  equilibrium condition} is satisfied at every $i\in V$:
\begin{equation}
\label{eq:equilibrium}
\omega(i) p(i)+\sum_{e=ij\in \delta(i)} \omega(e) p(j)=0.
\end{equation}
A stress $\omega\in \mathbb{R}^{V\cup E}$ is said to be {\em properly signed}  if 
$\omega({ij})\leq 0$ for $ij\in E\setminus \Sigma$ and $\omega({ij})\geq 0$ for $ij\in E\cap \Sigma$. 
A stress $\omega\in \mathbb{R}^{V\cup E}$ is said to be {\em PSD} if $\Omega\succeq 0$.
(Recall that $\Omega:=\sum_{i\in V} \omega(i) E_{ii}+ \sum_{ij\in E} \omega(ij)E_{ij}$.)
The following proposition is a well-known connection between equilibrium stresses and the PSD matrix completion problem (see, e.g., \cite{c,cg,lv}), which says that, for a given solution $X$ of (P), finding a dual solution that forms a complementarity pair with $X$ is equivalent to finding a properly signed equilibrium PSD stress.

\begin{proposition}
\label{prop:super}
Given a signed graph $(G,\Sigma)$ and  $X\in {\cal E}(G,\Sigma)$ of rank $d$ with a Gram matrix representation $X=P^{\top}P$,
let $c=\pi_G(X)$, and let $p:V(G)\rightarrow \mathbb{S}^{d-1}$ be defined by the column vectors of $P$.  
Then $(X, \Omega)$ forms a complementarity pair in ${\rm P}(G,\Sigma, c)$ if and only if  
$\omega$ is a properly signed equilibrium  PSD  stress of  spherical tensegrity $(G, p)$.
\end{proposition}
\begin{proof}
Note that $\omega$ satisfies the equilibrium condition (\ref{eq:equilibrium}) if and only if 
$P \Omega=0$.
If $\Omega\succeq 0$, the latter condition is equivalent to $\langle P^{\top} P, \Omega\rangle=0$, which is equivalent to 
$\sum_{i\in V} \omega_i+\sum_{ij\in E} \omega_{ij} c_{ij}=0$.
Thus $\Omega$ is dual optimal if and only if $\omega$ is a properly signed equilibrium PSD stress.
\end{proof}

 
Given a spherical tensegrity $(G,\Sigma, p)$ in $\mathbb{S}^{d-1}$, let $X={\rm Gram}(p)$
and $c=\pi_G(X)$.
$(G,\Sigma, p)$ is said to be {\em universally rigid} in $\mathbb{S}^{d-1}$ 
if ${\rm P}(G,\Sigma, c)$  has a unique solution $X$.
By specializing Proposition~\ref{prop:unique} and using the equivalence between (c6) and (c7), we have the following spherical version of the theorem by Connely-Gortler~\cite{cg}.

\begin{proposition}
\label{prop:universal}
Let $(G,\Sigma, p)$ be a spherical tensegrity in $\mathbb{S}^{d-1}$.
Then $(G,\Sigma, p)$ is universally rigid if and only if there is a sequence $\{ \omega^1, \dots. \omega^k\}$ of stresses  satisfying the following property: 
\begin{itemize}
\setlength{\parskip}{0.cm} 
 \setlength{\itemsep}{0.05cm}  
\item Each $\omega^i$ is properly signed.
\item $\{\Omega^1,\dots, \Omega^k\}$ is nested PSD.
\item For every $s$ with $1\leq s\leq k$,
\begin{equation}
\label{eq:semi-equilibrium}
\sum_{i\in V} \omega^s(i)+\sum_{ij\in E(G)} \omega^s(ij)( p(i)\cdot p(j))= 0. 
\end{equation}
\item The rank of the sequence $\{\Omega^1,\dots, \Omega^k\}$ is $n-d$.
\item There is no nonzero symmetric matrix $S\in {\cal S}^d$ satisfying 
$p(i) S p(j)=0$ for every  $ij\in V(G)\cup J_{\{\omega^1, \dots. \omega^k\}}$, 
where $J_{\{\omega^1, \dots. \omega^k\}}=\{ij\in E(G): \omega^s(ij)\neq 0 \text{ for some } s\}$.
\end{itemize}
\end{proposition}

\subsection{Examples}
\label{subsec:singularity_example}
In this subsection we shall discuss the singularity degree of two fundamental graph classes, complete graphs and cycles, as examples.
It is a well-known fact that ${\cal S}^n_+$ is an exposed convex set, that is, for each face $F$ there is a supporting hyperplane $H$ with $H\cap {\cal S}^n_+=F$. This fact implies the following proposition. We give a direct proof as it would be instructive.
\begin{proposition}
\label{prop:complete}
For any positive integer $n\geq 2$,  ${\sf sd}(K_n)=1$.
\end{proposition}
\begin{proof}
Let $c\in {\cal E}(K_n)$ and take a maximal rank $X\in {\rm P}(K_n, c)$. 
Take a base $\{u_1,\dots, u_k\}$  of the kernel of $X$, and let  $\Omega=\sum_{i=1}^k u_i u_i^{\top}$
and $\omega$ be the corresponding stress (i.e., $\omega(ij)=\Omega[i,j]$).
Then $\Omega\succeq 0$ and $\sum_{i\in V}\omega(i)+\sum_{ij\in E}\omega(ij) c(ij)=\langle \Omega, X\rangle=0$. The construction also implies $\rank X+\rank \Omega=n$.
Thus $\Omega$ satisfies (c2)(c3)(c4) (and (c1) since there is no inequality in the problem), and we get ${\sf sd}(K_n, c)\leq 1$ for any $c\in {\cal E}(K_n)$.

For any $X\in {\cal E}_n$ with ${\rm ker}\ X\neq \{0\}$, ${\rm P}(K_n, \pi_{K_n}(X))$ has a unique solution $X$, which is not positive definite. Hence ${\sf sd}(K_n)\geq 1$. 
\end{proof}

Given a graph $G$, an edge weight $c\in {\cal E}(G)$, and a sequence $\{\omega^1,\dots, \omega^k\}$ of stresses,
we say that a vertex $v$ is {\em stressed at the $i$-th stage} if $\omega^j(uv)\neq 0$ for some $j$ with $j\leq i$ and some $u\in N_G(v)\cup\{v\}$.

\begin{lemma}
\label{lem:cycle}
Let $C_n$ be the cycle of length $n$.
If $n\geq 4$, then ${\sf sd}(C_n)\geq 2$.
\end{lemma}
\begin{proof}
Let $V(C_n)=\{v_1,\dots, v_n\}$. 
We take $p: V(C_n)\rightarrow \mathbb{S}^{1}$ such that 
$p(v_1)=p(v_2)={\bf e}_1$, 
$p(v_n)={\bf e}_2$, 
and each $p(v_i)\ (3\leq i\leq n-1)$ lies on the interior of the spherical line segment between $p(v_1)$ and $p(v_n)$ in the ordering of $p(v_1)=p(v_2), p(v_3), \dots, p(v_{n-1}), p(v_n)$.
Let $X={\rm Gram}(p)$ and $c=\pi_{C_n}(X)$.

Note that ${\rm P}(C_n, c)$ has a unique solution $X$ (due to the metric inequality discussed in Subsection~\ref{subsec:K4}). Since $\rank X=2$,  at least $v_{n-1}$ or $v_n$  must be  stressed at the final stage of the sequence $\omega^1,\dots, \omega^h$ of the stresses obtained by the facial reduction to P$(C_n, c)$.
(Otherwise there would be a rank-three solution.)

Recall that $\omega^1$ satisfies the equilibrium condition (\ref{eq:equilibrium}) at every vertex.
The equilibrium condition at $v_2$ gives
$0=\omega^1(v_1v_2)p(v_1)+\omega^1(v_2)p(v_2)+\omega^1(v_2v_3)p(v_3)$.
Since $p(v_1)=p(v_2)$ while $p(v_2)$ and $p(v_3)$ are linearly independent, 
we have $\omega^1(v_2v_3)=0$.
Since $v_3$ has degree two, 
the equilibrium condition at $v_3$ further implies $\omega^1(v_3v_4)=0$.
Thus $v_3$ is not stressed at the first stage.
Continuing this, we conclude that neither $v_{n-1}$ nor $v_n$ is stressed at the first stage,
meaning that ${\sf sd}(C_n, c)> 1$.
\end{proof}
In Corollary~\ref{cor:K4} we will show that ${\sf sd}(C_n)=2$ if $n\geq 4$. 

\section{Characterizing Graphs with Singularity Degree at Most One}
\label{sec:degree_one}
In this section we give several facts, which will be useful for analyzing the singularity degree of graphs. As an application of these facts we give a characterization of graphs $G$ with ${\sf sd}(G)\leq 1$.

\subsection{Basic facts}

The following lemma will be a fundamental tool to analyze the behavior of the singularity degree.
\begin{lemma}
\label{lem:clique_sum}
Let $G$ be a clique sum of two undirected graphs $G_1$ and $G_2$.
Then \[{\sf sd}(G)\leq \max\{{\sf sd}(G_1), {\sf sd}(G_2)\}.\]
\end{lemma}
\begin{proof}
Take $c\in {\cal E}(G)$ with ${\sf sd}(G)={\sf sd}(G,c)$, and  let $c_1$ and $c_2$ be the restrictions of $c$ to $E(G_1)$ and $E(G_2)$, respectively.
For $i=1,2$ let $\omega^1_i,\dots, \omega^{h_i}_i$ be the sequence obtained by the facial reduction for P$(G_i, c_i)$, where $h_i={\sf sd}(G_i, c_i)$.
Each of these stresses can be regarded as a stress of $G$ by appending zero columns and zero rows.
Also we may assume that $h_1=h_2$ by adding zero stresses in the list.
Hence let $h=h_1$.
Then we define a sequence $\omega^1, \dots, \omega^{h}$ of stresses of $G$ by 
 $\omega^j=\omega_1^j+\omega_2^j \ (1\leq j\leq h)$.
 
 We show that the sequence $\{\omega^1,\dots, \omega^h\}$ satisfies (c1)(c2)(c3)(c4).
 (c1) and (c3) trivially hold.
 To see (c2), 
 let ${\cal V}^{j}=\{x\in \mathbb{R}^{V(G)}\mid x^{\top} \Omega^\ell x=0\ (\ell=1,\dots, j)\}$.
 ${\cal V}^{j}_1$ and ${\cal V}^{j}_2$ are similarly defined, using $\Omega^{\ell}_i$ in place of $\Omega^{\ell}$.
 By induction on $j\in [h]$, we first show that 
 \begin{equation}
 \label{eq:clique_sum}
 \text{for any $x\in {\cal V}^{j}$ and $i\in\{1,2\}$, 
 the restriction $x_i$ of $x$ to $V(G_i)$ belongs to ${\cal V}^{j}_i$.}
 \end{equation}
This is trivial for $j=0$.
 For $j\geq 1$,  note that $x\in {\cal V}^{j}\subseteq {\cal V}^{j-1}$
 and hence $x_i\in {\cal V}^{j-1}_i$ by induction.
 Since $\Omega^{j}_i$ is positive semidefinite on ${\cal V}^{j-1}_i$, we have 
 $x_i^{\top}\Omega^{j}_ix_i\geq 0$. 
Combining this with $x\in {\cal V}^{j}$, we get $0=x^{\top}\Omega^{j}x=\sum_{i=1,2} x_i^{\top} \Omega_i^{j} x_i \geq 0$, meaning that $x_i^{\top} \Omega_i^{j} x_i=0$. 
In other words,  (\ref{eq:clique_sum}) holds.

Now (\ref{eq:clique_sum}) implies that, 
for any $j$ and $x\in {\cal V}^{j-1}$, 
we have $x^{\top} \Omega^j x=\sum_{i=1,2} x^{\top}_i \Omega^j_ix_i\geq 0$.
Thus $\Omega^j$ is positive semidefinite on ${\cal V}^{j-1}$, and we have proved (c3).

To see (c4) we take a maximum rank solution $X_i$ of ${\rm P}(G_i,c_i)$,
and let $X_i=P_i^{\top}P_i$ be a Gram matrix representation of $X_i$, where each $P_i$ is row-independent.
Also let $S=V(G_1)\cap V(G_2)$.
Since $S$ induces a clique, we have $X_1[S,S]=X_2[S,S]$.
Hence each $P_i$ can be expressed as 
$P_i=\begin{array}{|c|c|} 
\hline
 \multirow{2}{*}{$\tilde{P}_i$} & 0 \\  \cline{2-2}
 & P_S  \\\hline
\end{array}$
for some row-independent matrix $P_S$ of size $d\times |S|$.
We concatenate $P_1$ and $P_2$ (by changing the row ordering of $P_2$ appropriately) such that 
\begin{equation}
\label{eq:sum}
P=\begin{array}{|c|c|c|} 
 \multicolumn{1}{r}{V(G_1)\setminus S} &  \multicolumn{1}{r}{S} &  \multicolumn{1}{r}{V(G_2)\setminus S } \\
\hline
 \multirow{2}{*}{$\tilde{P}_1$} & 0 & 0 \\  \cline{2-3}
 & P_S &   \multirow{2}{*}{$\tilde{P}_2$}  \\\cline{1-2}
 0 & 0 &\\ \hline
\end{array}.
\end{equation}
Then $P^{\top} P$ forms a feasible solution of ${\rm P}(G,\Sigma, c)$.

To show (c4), it suffices to prove that any $x\in {\cal V}^{h}$ is spanned by the rows of $P$.
Take any $x\in {\cal V}^{h}$. 
By  (\ref{eq:clique_sum}), $x_i\in {\cal V}_i^{h}$.
Hence $x_i$ is spanned by the rows of $P_i$.
Since $P_S$ is row independent, $x$ restricted to $\mathbb{R}^S$ is uniquely represented as a linear combination of the row vectors of $P_S$.
Hence, it follows from (\ref{eq:sum}) that the representation of $x_1$ as a linear combination of the row vectors of $P_1$ can be concatenated with that of $x_2$ as a linear combination of the rows of $P_2$ so that $x$ is represented as a  linear combination of the rows of $P$. In other words $x$ is spanned by rows of $P$.

This completes the proof.
\end{proof}



The next goal is to show Lemma~\ref{lem:induced_subgraph} which says that the singularity degree does not increase if we  take an induced subgraph.
To show this we need the following technical lemma, which will be also used in Section~\ref{sec:large}.
Recall that a vertex $v$ is said to be {\em stressed at the $i$-th stage} if $\omega^j(uv)\neq 0$ for some $j$ with $j\leq i$ and $u\in N_G(v)\cup\{v\}$.

\begin{lemma}
\label{lem:singularity_stress}
Let $(G,\Sigma)$ be a signed graph.
Let $\{\omega^1,\dots, \omega^k\}$ be the facial reduction sequence of P$(G, \Sigma, c)$, 
 $V_i$ be the set of vertices stressed at the $i$-th stage for $1\leq i\leq k$, and $V_0=\emptyset$. Then for each $i$ with $1\leq i \leq k$ the following hold.
\begin{itemize}
\setlength{\parskip}{0.1cm} 
 \setlength{\itemsep}{0.1cm} 
\item $\Omega^{i}[V\setminus V_{i-1}, V\setminus V_{i-1}]$ is positive semidefinite.
\item  For any $p:V(G)\rightarrow \mathbb{S}^{d-1}$ such that ${\rm Gram}(p)$ is feasible in ${\rm P}(G, \Sigma, c)$,
$\omega^{i}$ satisfies the "projected equilibrium condition" to the orthogonal complement of 
$\spa p(V_{i-1})$, that is,  
setting $\psi_{i-1}:\mathbb{R}^{d}\rightarrow (\spa  p( V_{i-1}) )^{\top}$ to be the  projection to the orthogonal complement of $\spa p(V_{i-1})$, we have
\begin{equation}
\label{eq:projected_eq}
\psi_{i-1}(p(v))\omega^{i}(v)+\sum_{e=uv\in \delta(v)} \psi_{i-1}(p(u))\omega^{i}(e)=0 \qquad (\forall v\in V). 
\end{equation}
\end{itemize}
\end{lemma}
\begin{proof}
%
The first claim trivially follows from the nested PSD-ness of $\{\Omega^1,\dots, \Omega^k\}$.

To see the second, let us take any $p$ such that ${\rm Gram}(p)$ is feasible.
Let $d'=\dim p(V_{i-1})$.
Let $\tilde{P}$ be the $(d-d')\times |V(G)|$-matrix whose $j$-th column is $\psi(p(j))$. 
By Proposition~\ref{prop:sequence_prop},
we have 
\begin{equation}
\label{eq:singular_stress}
\langle \Omega^i, \tilde{P}^{\top} \tilde{P}\rangle=0.
\end{equation}
Since each column of $\tilde{P}$ associated with $v\in V_{i-1}$ is zero and $\Omega^i[V\setminus V_{i-1}, V\setminus V_{i-1}]$ is positive semidefinite, (\ref{eq:singular_stress}) is equivalent to $\Omega^i \tilde{P}^{\top}=0$. Expanding $\Omega^i \tilde{P}^{\top}$, we have (\ref{eq:projected_eq}).
\end{proof}

\begin{lemma}
\label{lem:induced_subgraph}
Let $(G,\Sigma)$ be a signed graph and $H$ be an induced subgraph of $G$.
Then ${\sf sd}(G,\Sigma)\geq {\sf sd}(H, \Sigma \cap E(H))$.
\end{lemma}
\begin{proof}
Take $c_H\in {\cal E}(H, \Sigma\cap E(H))$ such that 
${\sf sd}(H,\Sigma\cap E(H))={\sf sd}(H,\Sigma\cap E(H),c_H)$,
and take $p: V(H)\rightarrow \mathbb{S}^d$ for some $d\geq 0$ such that 
${\rm Gram}(p)$ is a maximum rank feasible solution of ${\rm P}(H,\Sigma\cap E(H),c_H)$.

Denote the vertices of $V(G)\setminus V(H)$ by $\{v_1, \dots, v_k\}$, where $k=|V(G)\setminus V(H)|$.
By rotation in $\mathbb{R}^n$, we may assume that each ${\bm e}_i$ is orthogonal to   $\spa p(V(H))$ for $1\leq i\leq k$.
We extend $p$ such that 
$p(v_i)={\bm e}_{i}$  for each $v_i\in V(G)\setminus V(H)$.
Let $c_G=\pi_G({\rm Gram}(p))$ for the extended $p$.

Take the sequence $\omega^1, \dots, \omega^h$ of stresses of $G$ obtained by the facial reduction of ${\rm P}(G, \Sigma, c_G)$.
We first show that 
\begin{equation}
\label{eq:induced_subgraph}
\text{each $v_i\in V(G)\setminus V(H)$ is not stressed at the $j$-th stage}
\end{equation}
by induction on $j$ from $j=0$ through $j=h$.
The claim is trivial for $j=0$ and hence assume $j>0$.
Take any $v_i\in V(G)\setminus V(H)$.
By induction, ${\bm e}_i$ is orthogonal to $\spa p(V_{j-1})$, 
where $V_{j-1}$ is the set of vertices stressed at the $(j-1)$-th stage.
Hence a projected equilibrium condition (\ref{eq:projected_eq}) at $v_i$ (Lemma~\ref{lem:singularity_stress})  implies that
\begin{equation}
\label{eq:4-3}
\left(\sum_{u\in N_G(v_i)\cup \{v_i\}}\omega^j(uv_i)p(u)\right)\cdot {\bm e}_\ell=0
\end{equation}
for $1\leq \ell\leq k$.
The definition of $p$ further implies that $p(v_{\ell})\cdot {\bm e}_{\ell}=1$ and $p(u)\cdot {\bm e}_{\ell}=0$ if $u\neq v_{\ell}$. By (\ref{eq:4-3}), we obtain
$\omega^j(v_{\ell} v_i)=0$ for every $v_{\ell}\in( N_G(v_i)\cup \{v_i\})\setminus V(H)$.
Similarly, for $w\in N_G(v_i)\cap V(H)$, 
a projected equilibrium condition at $w$ implies 
\[
\omega^j(wv_i)= \left(\sum_{u\in N_G(w)\cup \{w\}}\omega^j(wu)p(u)\right)\cdot {\bm e}_i=0.
\]
Thus $v_i$ is not stressed at the $j$-th stage.

By (\ref{eq:induced_subgraph}), $\omega^j$ is supported on $E(H)$.
However, since $H$ is induced, the restriction of $\omega^1, \dots, \omega^h$ to $V(H)\cup E(H)$ gives a sequence of stresses of $H$ which satisfies (c1)(c2)(c3)(c4) for P$(H, \Sigma\cap E(H), c_H)$.
Thus ${\sf sd}(G,\Sigma)\geq {\sf sd}(G,\Sigma, c)\geq {\sf sd}(H,\Sigma\cap E(H),c_H)={\sf sd}(H,\Sigma\cap E(H))$.
\end{proof}

\subsection{Graphs with Singularity Degree at Most One}
%

Before going to the characterization of graphs with singularity degree at most one, we first remark on graphs with  singularity degree zero. 
\begin{theorem}
\label{thm:sd0}
Let $G$ be a graph. 
Then ${\sf sd}(G)=0$ if and only if $E(G)=\emptyset$.
\end{theorem}
\begin{proof}
If $G$ has no edge, then any positive definite matrix of size $|V(G)|$ is a solution of the completion problem for $G$. Hence ${\sf sd}(G)=0$.
If $G$ has an edge, then Proposition~\ref{prop:complete} and Lemma~\ref{lem:induced_subgraph} implies ${\sf sd}(G)\geq 1$. 
\end{proof}

An induced cycle of length at least four is called a {\em hole}.
A graph is said to be {\em chordal} if it has no hole.

\begin{theorem}
\label{thm:sd1}
Let $G$ be a graph.
Then ${\sf sd}(G)\leq 1$ if and only if $G$ is chordal.
\end{theorem}
\begin{proof}
%
Suppose that $G$ is chordal. 
It is known that any chordal graph is the clique sum of smaller chordal graphs if it is not complete.
Therefore we have ${\sf sd}(G)\leq 1$ by Proposition~\ref{prop:complete} and Lemma~\ref{lem:clique_sum}.

Conversely suppose that $G$ is not chordal.
Then $G$ has a hole $H$. 
By Lemma~\ref{lem:cycle} and Lemma~\ref{lem:induced_subgraph}, we get
${\sf sd}(G)\geq {\sf sd}(H)\geq 2$.
\end{proof}

The sufficiency of Theorem~\ref{thm:sd1} was already shown by Druvyatskiy, Pataki and Wolkowicz~\cite{dpw} by a different approach. 
In fact they showed that, at the general model of looped graphs, ${\sf sd}(G,c)=1$ holds for any $c:E\rightarrow \mathbb{R}$  if the subgraph induced by the looped vertices is chordal.
They also posed a question about characterizing graphs $G$ for which ${\sf sd}(G,c)=1$ for any $c$. 
Although their question is given at the general model, this generality does not make any difference when  constructing an example with large singularity degree. 
Thus our lower-bound construction  also gives a characterization in their general model.

\section{Nondegenerate Singularity Degree}
\label{sec:nondegenerate}
Recall that $c:E\rightarrow [-1,1]$ is called nondegenerate if $c(ij)\notin \{-1,1\}$ for every $ij\in E$.
Theorem~\ref{thm:sd1} gives a characterization of undirected graphs with singularity degree at most one. In the proof, we have shown that, if a graph $G$ contains a hole, then there is a degenerate edge weight $c$ for which ${\sf sd}(G,c)\geq 2$. 
However, from the optimization view point, such degenerate edges can be easily  eliminated at the preprocessing phase  by  edge-contraction, and it would be natural to look at singularity degree over nondegenerate edge weights $c$. 
Thus we define the  {\em nondegenerate singularity degree} ${\sf sd}^*(G)$ of a graph $G$ by 
\[
{\sf sd}^*(G):=\max\{{\sf sd}(G,c)\mid c\in {\cal E}(G)\cap (-1,1)^{E(G)}\}.
\]
In this section we shall give a characterization of undirected graphs with nondegenerate singularity degree at most one.

\subsection{Bounding the nondegenerate singularity degree of $K_4$-minor free graphs}
\label{subsec:K4}
A key special case for bounding the nondegenerate singularity degree is the class of $K_4$-minor free graphs.
It is known that any $K_4$-minor free graph contains a cut of size two.
Hence a promising approach would be to split the graph into two parts along the minimum cut, 
and combine the primal and dual pair of each smaller instances obtained by induction to get a desired primal and dual pair.
This strategy turns out to be nontrivial even if the size of the minimum cut is two, 
as we have to combine solutions keeping the positivity as well as the complementarity. 
To keep the positivity, a key idea from rigidity theory (especially, from \cite{c11}) is to look at the signed version of the PSD matrix completion problem so that  the sign of each entry in dual solutions is controlled.
 Thus we prove a more general statement on the nondegenerate singularity degree of signed graphs. 
Since the detail is rather involved and technical, in this paper we shall only give the statement, and the full proof is given  in a separate paper~\cite{t}.

To state the main result of \cite{t}, we need to explain minors in signed graphs.
Recall that a signed graph $(G,\Sigma)$ is a pair of an undirected graph $G$ (which may contain parallel edges) and  $\Sigma\subseteq E(G)$.
An edge in $\Sigma$ (resp.~in $E(G)\setminus \Sigma$) is called {\it odd} (resp.~even),
and a cycle (or a path) is said to be {\it odd} (resp. {\it even}) if the number of odd edges in it is odd (resp.~even).
The {\em resigning} on  $X\subseteq V$ changes $(G,\Sigma)$ with $(G,\Sigma\Delta \delta(X))$, where 
$A \Delta B:=(A\setminus B)\cup (B\setminus A)$ for any two sets $A, B$.
Two signed graphs are said to be {\em (sign) equivalent} if they can be converted to each other by a series of resigning operations.
A signed graph is called a {\it minor} of $(G,\Sigma)$ if it can be obtained by a sequence of the following three operations:
(i) the removal of an edge,
(ii) the contraction of an even edge, and
(iii) resigning.
We say that $(G,\Sigma)$ is {\em $(H, \Sigma')$ minor free}  if $(H, \Sigma')$ is not a minor of $(G,\Sigma)$.

For an undirected graph $H$,  signed graph $(H, E(H))$ is called  {\em odd-$H$}.
%

For a signed graph $(G,\Sigma)$, $c\in [-1,1]^E$ is said to be {\em nondegenerate}
if $c(e)\neq 1$ for every $e\in E\setminus \Sigma$ and $c(e)\neq -1$ for every $e\in \Sigma$. 
The following is the main theorem of \cite{t}.
\begin{theorem}[Theorem 3.1 in \cite{t}]
\label{thm:oddK4}
Let $(G,\Sigma)$ be an odd-$K_4$-minor free signed graph and  
$c$ be nondegenerate.
If $c\in {\cal E}(G, \Sigma)$, then ${\rm P}(G,\Sigma,c)$ has a dual solution 
that satisfies the strict complementarity condition with a maximum rank primal solution.
\end{theorem}

The nondegenerate singularity degree of signed graphs can be defined, in the same manner,  
by ${\sf sd}^*(G, \Sigma):=\max\{{\sf sd}(G, \Sigma, c) \mid \text{ nondegenerate } c\in {\cal E}(G) \}$.
The following is a restatement of Theorem~\ref{thm:oddK4} in terms of the singularity degree.
\begin{corollary}
\label{cor:oddK4}
Let $(G,\Sigma)$ be an odd-$K_4$-minor free signed graph.
Then ${\sf sd}^*(G,\Sigma)\leq 1$.
\end{corollary}
This also implies the following.
\begin{corollary}
\label{cor:K4}
Let $G$ be a $K_4$-minor free  graph.
Then ${\sf sd}^*(G)\leq 1$.
\end{corollary}
\begin{proof}
Take any nondegenerate $c\in {\cal E}(G)\cap (-1,1)^{E(G)}$.
Let $(H, \Sigma_H)$ be  the signed graph obtained from $G$ by replacing each edge by two parallel edges with distinct sign.
Note that this signed graph is odd-$K_4$ minor free.

Define $c'\in [-1,1]^{E(H)}$ such that $c'(e)$ is equal to the edge weight $c$ assigned to the edge of $G$ corresponding to $e$.
Then $c'\in {\cal E}(H,\Sigma)$ since a solution of ${\rm P}(G, c)$ is a solution of ${\rm P}(G, \Sigma, c')$.
Hence by Theorem~\ref{thm:oddK4} there is a dual solution of  ${\rm P}(H, \Sigma_H, c')$ that satisfies the strict complementarity condition. 
This dual solution can be regarded as a dual solution of ${\rm P}(G, c)$ which satisfies the strict complementarity condition.
\end{proof} 

%
%

A key observation to prove Theorem~\ref{thm:oddK4} is to establish a signed version of 
Laurent's theorem on a characterization of ${\cal E}(G)$ in terms of the metric polytope of $G$.
More specifically, the {\em metric polytope} of an undirected graph $G$ is defined by
\begin{equation*}
{\rm MET}(G)=
\left\{x\in [0,1]^{E} \ \Bigg| \ \sum_{e\in E(C)\setminus F} x(e) - \sum_{e\in  F} x(e)
\geq 1-|F| : 
\begin{array}{cc} \text{$\forall$ cycle $C$ in $G$} \\ \text{$\forall F\subseteq C$: $|F|$ is odd} \end{array}
\right\}.
\end{equation*}
Throughout the discussion we set ${\rm arccos}: [-1,1]\rightarrow [0, \pi]$. 
Then it was shown by Laurent~\cite{l97} that ${\rm arccos}({\cal E}(G))/\pi\subseteq {\rm MET}(G)$, with equality if and only  if $G$ is $K_4$-minor free.

Suppose that we are given a signed graph $(G,\Sigma)$.
Each signing $\Sigma$ defines  a sign function $\sigma:E(G)\rightarrow \{-1, +1\}$ such that $\sigma(e)=-1$ if and only if $e\in \Sigma$.
We define a signed version of the metric polytope as follows:
\begin{equation}
\label{eq:met}
{\rm MET}(G,\Sigma)=
\left\{x\in [0,1]^E \ \Bigg| \ \sum_{e\in E(C)} \sigma(e)x(e) 
\geq 1-|E(C)\cap \Sigma| : 
\forall \text{odd cycle $C$ in $(G,\Sigma)$ }
\right\}
\end{equation}
where a cycle is said to be {\em odd} if the number of odd edges in it is odd.
To prove Theorem~\ref{thm:oddK4}, we have shown the following generalization of Laurent's theorem.
\begin{theorem}[Theorem 3.2 in \cite{t}]
\label{thm:signed}
Let $(G, \Sigma)$ be a signed graph Then ${\rm arccos}({\cal E}(G,\Sigma))/\pi\subseteq {\rm MET}(G, \Sigma)$, with equality if and only  if $(G, \Sigma)$ is odd-$K_4$-minor free.
\end{theorem}
This theorem suggests an exact connection  between supporting hyperplanes of ${\rm MET}(G, \Sigma)$ and those of ${\cal E}(G, \Sigma)$ if $(G, \Sigma)$ is odd-$K_4$ minor free.
In the proof of Theorem~\ref{thm:oddK4} in \cite{t} we make this connection explicit.
Specifically, we say that an odd cycle $C$ is {\em tight} with respect to $c$ if 
\[
\sum_{e\in E(C)} \sigma(e)({\rm arccos}(c(e))/\pi)=1-|E(C)\cap \Sigma|.
\]
This {\em metric} equation uniquely determines a primal solution of the completion problem restricted to $C$,
and if $c$ is nondegenerate there is a unique dual solution $\Omega_C$  (up to scaling) which is supported on $E(C)$ and satisfies the strict complementarity condition with the unique solution.
Let $\Omega=\sum_C \Omega_C$, where the sum is taken over all tight odd cycles (each $\Omega_C$ is regarded as a $n\times n$ matrix by appending zero rows and columns).
It was shown in \cite{t} that, if $(G,\Sigma)$ is odd-$K_4$ minor free, then there always exists a primal solution with rank equal to $n-\rank \Omega$, and hence the supporting hyperplane of the primal feasible region is  determined by $\Omega$.

By regarding each unsigned graph as a singed graph with double edges (as explained in the proof of Corollary~\ref{cor:K4}), the observation can be applied even to (unsigned)$K_4$-minor free graphs. Namely,  the supporting hyperplane of the primal feasible region is  determined by computing tight cycles, where in the unsigned case a cycle $C$ is said to be {\em tight} (with respect to $c$) if $\sum_{e\in E(C)\setminus F} x(e)-\sum_{e\in F} x(e)\geq 1-|F|$ for some $F\subseteq C$ with odd $|F|$. 

\subsection{Bounding nondegenerate singularity degree}
\label{subsec:nondegenerate}

For analyzing nondegenerate singularity degree,  we first remark that the proofs 
of Lemma~\ref{lem:clique_sum} and Lemma~\ref{lem:induced_subgraph} can be adapted to show the corresponding statements for ${\sf sd}^*(G)$, implying the following. 
(We omit the identical proofs.)
\begin{lemma}
\label{lem:nondegenerate_clique_sum}
Let $G$ be a clique sum of two undirected graphs $G_1$ and $G_2$.
Then \[{\sf sd}^*(G)\leq \max\{{\sf sd}^*(G_1), {\sf sd}^*(G_2)\}.\]
\end{lemma}
\begin{lemma}
\label{lem:nondegenerate_induced_subgraph}
Let $G$ be an undirected graph and $H$ be an induced subgraph of $G$.
Then ${\sf sd}^*(G)\geq {\sf sd}^*(H)$.
\end{lemma}

Before going to our main theorem, we first prove the following easy observation.
\begin{theorem}
\label{thm:sd*0}
Let $G$ be a graph.
Then ${\sf sd}^*(G)=0$ if and only if $G$ has no cycle (i.e., $K_3$-minor free).
\end{theorem}
\begin{proof}
Suppose that $G$ has no cycle. 
Then ${\sf sd}^*(G)=0$ holds by Proposition~\ref{prop:complete} and Lemma~\ref{lem:nondegenerate_clique_sum}.

In view of Lemma~\ref{lem:nondegenerate_induced_subgraph}, the converse follows by showing  ${\sf sd}^*(C_n)\geq 1$ for cycle $C_n$ with $n\geq 3$.
To see this, take $p:V(C_n)\rightarrow \mathbb{S}^1$ such that $p(v_1)={\bf e}_1$, $p(v_n)={\bf e}_2$, and each $p(v_i)\ (2\leq i\leq n-1)$ lies on the interior of the spherical line segment between $p(v_1)$ and $p(v_n)$ in the ordering of indices. Setting  $X={\rm Gram}(p)$ and $c=\pi_{C_n}(X)$, 
the cycle $C_n$ is tight with respect to $c$, and ${\rm P}(C_n, c)$ has a unique solution  $X$.
Since $X$ is not positive definite and $c$ is nondegenerate, we obtain ${\sf sd}^*(C_n)\geq 1$.
\end{proof}

To describe the characterization of graphs with nondegenerate singularity degree at most one, we need some notation.The inverse operation of an edge-contraction is called a {\em (vertex) splitting} operation.
If a graph $G'$ can be constructed from $G$ by a sequence of splitting operations, $G'$ is said to be a {\em splitting} of $G$.
Moreover, if $G'\neq G$, $G'$ is said to be a {\em proper} splitting of $G$.
{\em Subdividing} an edge $e=uv$ means inserting a new vertex $w$ and replace $e$ with new edges $uw$ and $wv$.
If $G'$ can be constructed from $G$ by subdividing edges of $G$, then 
$G'$ is called a {\em subdivision} of $G$.
Note that splitting a node of degree two or three amounts to subdividing an edge incident to it.

A {\em wheel} $W_n$ is an undirected graph with $n$ vertices formed by connecting a vertex to all vertices of the cycle of length $n-1$.
See Figure~\ref{fig:Wn}(a). 
The vertex adjacent to all other vertices in $W_n$ is called the {\em center vertex}.

Now we are ready to state our main theorem.
\begin{theorem}
\label{thm:nonsingular}
The following assertions are equivalent for an undirected graph  $G$:
\begin{description}
\setlength{\parskip}{0.cm} 
 \setlength{\itemsep}{0.05cm} 
\item[(i)] ${\sf sd}^*(G)\leq 1$.
\item[(ii)] $G$ has neither $W_n\ (n\geq 5)$ nor a proper splitting of $W_n\ (n\geq 4)$ as an induced subgraph.
\item[(iii)] $G$ can be constructed by clique sums of complete graphs and $K_4$-minor free graphs.
\end{description}
\end{theorem}

We should remark that the graph class defined by the property of (ii) has been already studied in the context of the PSD completion problem as the family of graphs $G$ whose ${\cal E}(G)$ can be characterized by the metric inequalities and the clique inequalities~\cite{bjl}.
The equivalence between (ii) and (iii) was shown in \cite{jm}.
 (iii) $\Rightarrow$ (i) of Theorem~\ref{thm:nonsingular} follows from  Proposition~\ref{prop:complete}, Corollary~\ref{cor:K4}, and Lemma~\ref{lem:nondegenerate_clique_sum}.
To see (i)$\Rightarrow$(ii) of Theorem~\ref{thm:nonsingular} it suffices to show the following Lemma~\ref{lem:nonsingular}.
\begin{lemma}
\label{lem:nonsingular}
Suppose that $G$ contains $W_n$ for some $n\geq 5$ or a proper splitting of $W_n$ for some $n\geq 4$ as an induced subgraph.
Then ${\sf sd}^*(G)\geq 2$.
\end{lemma}

By Lemma~\ref{lem:nondegenerate_induced_subgraph}, Lemma~\ref{lem:nonsingular} follows by showing ${\sf sd}^*(G)\geq 2$ if 
$G$ is isomorphic to $W_n$ for some $n\geq 5$ or a proper splitting of $W_n$ for some $n\geq 4$.
We first solve the case when $G$ is isomorphic to $W_n$ in Lemma~\ref{lem:Wn}, and then extend the proof idea to the case when $G$ is a proper splitting of $W_n$.
For simplicity of notation, throughout the following discussion, 
${\rm P}(G,\pi_G({\rm Gram}(p)))$ is denoted by ${\rm P}(G,p)$ for each spherical framework $(G,p)$.

\begin{figure}
\centering
\begin{minipage}{0.3\textwidth}
\centering
\includegraphics[scale=0.5]{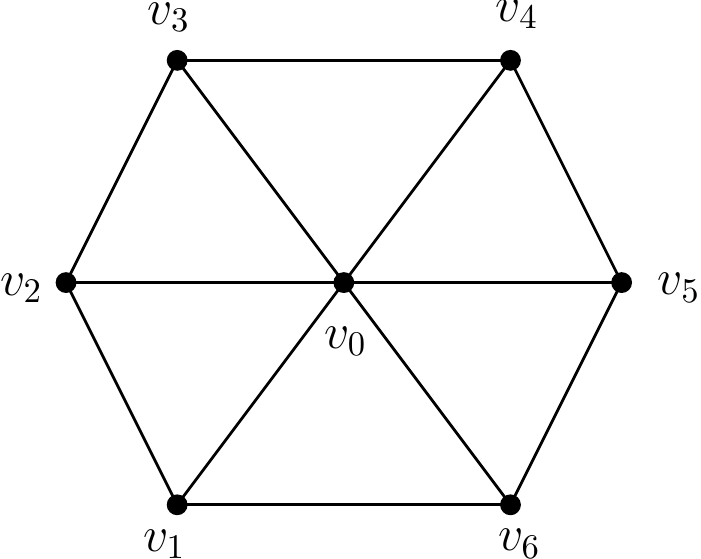}
\par
(a)
\end{minipage}
\begin{minipage}{0.36\textwidth}
\centering
\includegraphics[scale=0.4]{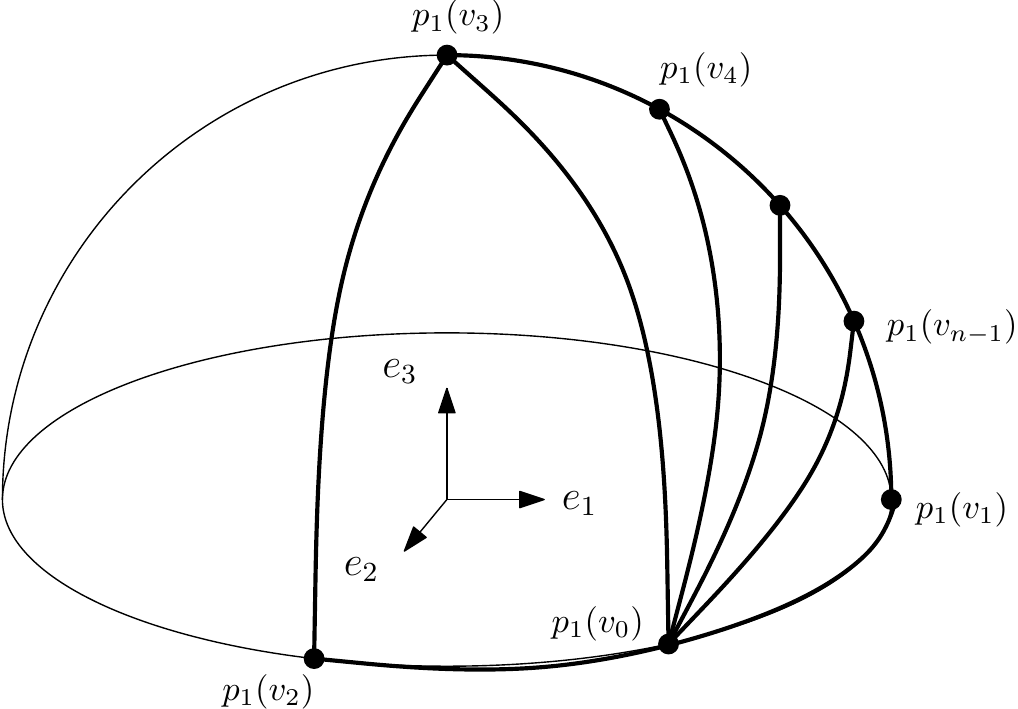}
\par
(b)
\end{minipage}
\begin{minipage}{0.28\textwidth}
\centering
\includegraphics[scale=0.6]{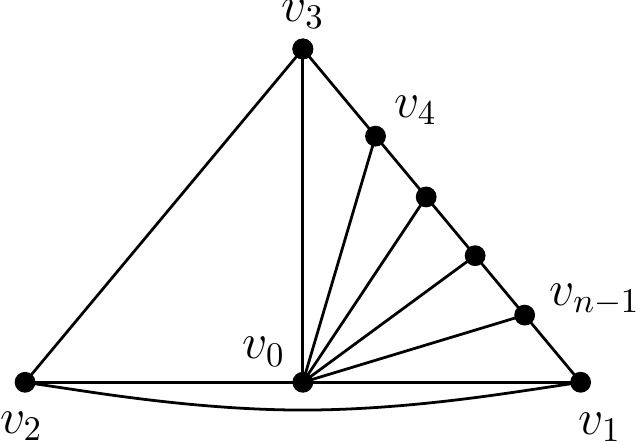}
\par
(c)
\end{minipage}
\caption{(a) $W_n$, (b) $(W_n, p_1)$, (c) a projection of $(W_n, p_1)$ to the plane
(representing the point-line incidence).}
\label{fig:Wn}
\end{figure}

\begin{lemma}
\label{lem:Wn}
${\sf sd}^*(W_n)\geq 2$ if $n\geq 5$.
\end{lemma}
\begin{proof}
Denote  $V(W_n)=\{v_0, v_1,\dots, v_{n-1}\}$ such that $v_0$ is the center vertex and 
$v_1,\dots, v_{n-1}$ are linked in this ordering.
We consider $p_1:V(W_n)\rightarrow \mathbb{S}^2$ defined as follows:
\begin{itemize}
\setlength{\parskip}{0.cm} 
 \setlength{\itemsep}{0.05cm} 
\item $p_1(v_i)={\bf e}_i$ for $i=1,2,3$, and $p_1(v_0)$ is the midpoint of the (spherical) line segment between $p_1(v_1)$ and $p_1(v_2)$.
\item $p_1(v_4), \dots, p_1(v_{n-1})$ are aligned on the spherical line segment between $p_1(v_3)$ and $p_1(v_1)$ in this ordering so that the length of the line segment between two consecutive vertices is nonzero.
\end{itemize}
See Figure~\ref{fig:Wn}.
We first show that ${\rm P}(W_n, p_1)$ has a unique solution, which is ${\rm Gram}(p_1)$.
To see this, take any solution ${\rm Gram}(q)$ of ${\rm P}(W_n, p_1)$. 
$\{p_1(v_0), p_1(v_1), p_1(v_2)\}$ is congruent to   $\{q(v_0), q(v_1), q(v_2)\}$ since $\{v_0,v_1,v_2\}$ forms a tight cycle (with respect to $c=\pi_G({\rm Gram}(p))$).
Since the span of $\{q(v_0), q(v_1), q(v_2)\}$ forms a line but $v_3$ is connected to $\{v_0, v_1, v_2\}$ by two edges, $q(v_3)$ is uniquely determined (up to orthogonal transformation).
Therefore, $q(v_3)\cdot q(v_1)$ is uniquely determined, which is equal to zero. 
Hence the solution set of ${\rm P}(W_n, p_1)$ is equal to that of 
${\rm P}(W_n+v_1v_3, p_1)$.
In the latter problem,  $\{v_1, v_3, v_4, \dots, v_{n-1}\}$ forms a tight cycle, and hence 
the remaining positions of $q$ are also unique.
Thus ${\rm P}(W_n, p_1)$ has a unique solution.
This in turn implies that at least $p_1(v_{3})$ or $p_1(v_{4})$ must be stressed at the final stage of the facial reduction, since otherwise we would have a solution of rank more than $\rank {\rm Gram}(p_1)$.

Let $\omega$ be the first stress of the facial reduction to  ${\rm P}(W_n, p_1)$.
Since $\{p_1(v_0), p_1(v_1), p_1(v_2)\}$ lines on ${\rm span}\{{\bf e}_1, {\bf e}_2\}$ while $p_1(v_3)$ is not,  
the equilibrium condition~(\ref{eq:equilibrium})  at $v_2$ implies  $\omega(v_2v_3)=0$.
Hence, the equilibrium condition at $v_3$ implies $\omega(v_0v_3)=0$ and $\omega(v_3v_4)=0$.
Now,  since $p_1(v_3), \dots, p_1(v_{n-1}), p_1(v_1)$ lie on a spherical line, 
applying the equilibrium condition from $v_4$ through $v_{n-1}$,
we get that $v_3$ and $v_{4}$ are not stressed at the first stage.
This in turn implies ${\sf sd}^*(W_n)\geq 2$.
\end{proof}

The key of the proof of Lemma~\ref{lem:Wn} is the construction of $p_1$ for which ${\rm P}(W_n, p_1)$ has a unique solution.
In order to apply the above proof idea to a subdivision or/and a splitting of $W_n$, we next examine the unique solvability of ${\rm P}(G_4,p_4)$ for a spherical framework $(G_4, p_4)$ defined below. 
The construction of $(G_4, p_4)$ is best explained by first looking at smaller examples $(G_2, p_2)$ and $(G_3, p_3)$ defined as follows.
\begin{itemize}
\setlength{\parskip}{0.cm} 
 \setlength{\itemsep}{0.05cm} 
\item $G_2$ is defined as the graph obtained from $K_4$ by subdividing edge $v_0v_3$ once,
and the new vertex is denoted by $w$.
\item $p_2(v_1)={\bf e}_1, p_2(v_2)={\bf e}_2, p_2(w)={\bf e}_3$.
\item $p_2(v_0)$ is the midpoint of the spherical line segment between $p_2(v_1)$ and $p_2(v_2)$.
\item $p_2(v_3)$ is the midpoint of the spherical line segment between $p_2(v_0)$ and $p_2(w)$.
\end{itemize}
See Figure~\ref{fig:G2}(a).

\begin{figure}
\centering
\begin{minipage}{0.4\textwidth}
\centering
\includegraphics[scale=0.55]{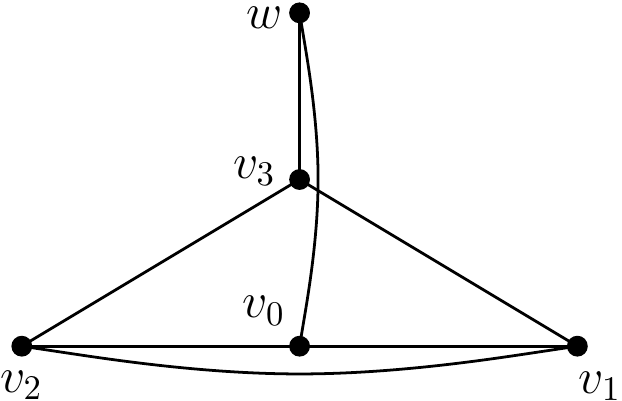}
\par
(a) $(G_2, p_2)$
\end{minipage}
\begin{minipage}{0.4\textwidth}
\centering
\includegraphics[scale=0.55]{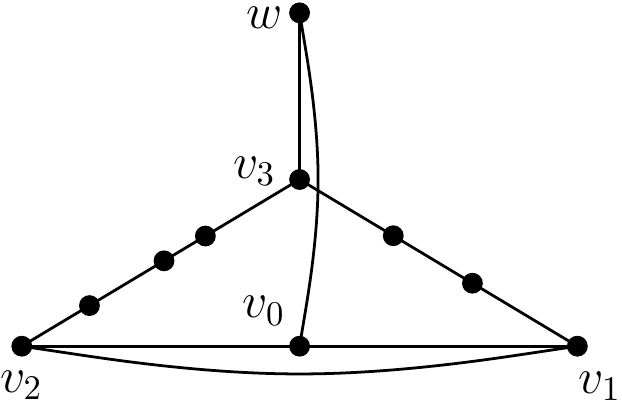}
\par
(b) $(G_3, p_3)$
\end{minipage}
\caption{(a) A projection of $(G_2, p_2)$ and (b) that of $(G_3, p_3)$ to the plane.}
\label{fig:G2}
\end{figure}

\begin{lemma}
\label{lem:G2}
${\rm P}(G_2, p_2)$ has a unique solution.
\end{lemma}
\begin{proof}
Observe that the problem has a unique solution if it is restricted to $\{v_0, v_1, v_2, v_3\}$.
(This  can be checked by elementary geometry, but
we can also apply Theorem~5.3 given in \cite{t} to confirm this observation.)
This means that $q(v_0)\cdot q(v_3)=p_2(v_0)\cdot p_2(v_3)$ for any solution ${\rm Gram}(q)$ of ${\rm P}(G_2, p_2)$. 
Note that, in ${\rm P}(G_2+v_0v_3, p_2)$, $\{v_0, v_3, w\}$ forms a tight cycle.
Hence $q$ is uniquely determined.
\end{proof}

Next we consider a spherical framework $(G_3,p_3)$ obtained from $(G_2, p_2)$ 
by a sequence of subdivisions  of the spherical line segments 
$p_2(v_3)p_2(v_1)$ and $p_2(v_3)p_2(v_2)$.
See Figure~\ref{fig:G2}(b).
\begin{lemma}
\label{lem:G3}
${\rm P}(G_3, p_3)$ has a unique solution.
\end{lemma} 
\begin{proof}
Take  any solution ${\rm Gram}(q)$ of ${\rm P}(G_3, p_3)$.
By the construction of $(G_3, p_3)$,  we have $q(v_3)\cdot q(v_i)\geq p_3(v_3)\cdot p_3(v_i)$ for $i=1,2$.
This implies $q(v_0)\cdot q(v_3)\geq p_3(v_0)\cdot p_3(v_3)$, where 
the equality holds if and only if $q(v_3)\cdot q(v_i)= p_3(v_3)\cdot p_3(v_i)$ for all $i=1,2$.
Since $q(v_0)\cdot q(v_3)\leq p_3(v_0)\cdot p_3(v_3)$ holds due to the metric inequality on 
$\{v_0, v_3, w\}$ (which follows because $G_3$ contains  $v_0w$ and $wv_3$), we obtain that $q(v_0)\cdot q(v_3)= p_3(v_0)\cdot p_3(v_3)$,
and that $q(v_3)\cdot q(v_i)= p_3(v_3)\cdot p_3(v_i)$ for $i=1,2$.
The lemma now follows from Lemma~\ref{lem:G2}.
\end{proof}

Now consider a subdivision $G_4$ of the complete graph $K_4$. 
We denote   the center vertex by $v_0$ and   the other three vertices of $K_4$ by $\{v_1, v_2, v_3\}$ as shown in Figure~\ref{fig:G4}(a).
Take any edge $w_1w_2$ in the path between $v_1$ and $v_2$ such that   $w_1$ is  closer to $v_1$ than $w_2$.
Also let $w_3$ be the vertex next to $v_0$ on the path from $v_0$ to $v_3$.
(Note that $w_i$ may be $v_i$.)
See Figure~\ref{fig:G4}(a).
We define $(G_4, q_4)$ as follows:
\begin{itemize}
\setlength{\parskip}{0.cm} 
 \setlength{\itemsep}{0.05cm} 
\item $p_4(w_1)={\bf e}_1, p_4(w_2)={\bf e}_2, p_4(w_3)={\bf e}_3$.
\item $p_4(v_0)$ is the midpoint of the spherical line segment between $p_4(w_1)$ and $p_4(w_2)$.
\item  If $v_3\neq w_3$, $p_4(v_3)$ is the midpoint of the spherical line segment between $p_4(v_0)$ and $p_4(w_3)$.
\item If $v_i\neq w_i\ (i=1,2)$, $p_4(v_i)$ lies on the spherical line segment between $p_4(v_0)$ and $p_4(w_i)$ such that the distance between $p_4(w_i)$ and $p_4(v_i)$ is equal to $\epsilon$ for some small 
$\epsilon>0$.
\item For other vertices $u$ (having degree two), $p_4(u)$ is placed by subdividing the corresponding spherical line segment.
\end{itemize}
See Figure~\ref{fig:G4}(b).
By Lemma~\ref{lem:G3} and the metric inequality,  we have the following.
\begin{lemma}
\label{lem:G4}
${\rm P}(G_4, p_4)$ has a unique solution.
\end{lemma}

We are now ready to prove Lemma~\ref{lem:nonsingular}.
\begin{proof}[Proof of Lemma~\ref{lem:nonsingular}]
Take an inclusionwise minimal induced subgraph $H$ of $G$ which is either $W_n \ (n\geq 5)$ or a proper splitting of $W_n\ (n\geq 4)$.
It can be easily shown by the minimality of $H$ that 
$H$ is (i) a subdivision of $W_n (n\geq 4)$ with $H\neq K_4$ or 
(ii) a subdivision of the graph obtained from $W_n (n\geq 4)$ by splitting the center vertex once.
In view of Lemma~\ref{lem:nondegenerate_induced_subgraph}, it suffices to show the statement for $G=H$.

\medskip
\noindent
Case 1: Suppose that $G$ is a subdivision of $W_n\ (n\geq 4)$ with $G\neq K_4$. 
Denote $G$ by $G_5$ in this case, and we show ${\sf sd}^*(G_5)\geq 2$.
We use $\{u_0, u_1, \dots, u_{n-1}\}$ to denote the vertex of degree more than two in $G_5$ as shown in Figure~\ref{fig:G4}(c), where 
$u_0$ denotes the center vertex.
By identifying $u_i$ with $v_i$ for $i=0, 1,2,3$,
$G_5$ contains  a subdivision $G_4$ of $K_4$.
We extend spherical framework $(G_4,p_4)$ to $(G_5, p_5)$ 
by realizing the remaining paths (i.e., the paths from $u_0$ to $u_i$ for $4\leq i \leq n-1$) by straight line segments as shown in Figure~\ref{fig:G4}(d). 
Since $G_5$ is a subdivision of $W_n\ (n\geq 4)$ with $G_5\neq K_4$, 
 there is at least one vertex $v^*$ on the path between $u_1$ and $u_3$, (either $u_4$ or a vertex obtained by a subdivision).
Since $(G_5, p_5)$ contains $(G_4, p_4)$ as a subframework and the remaining paths are realized by straight lines, 
${\rm P}(G_5, p_5)$ has a unique solution by Lemma~\ref{lem:G4}.
This implies that at least $u_3$ or $v^*$ is stressed at the final stage of the facial reduction to ${\rm P}(G_5, p_5)$.
However, applying the analysis of Lemma~\ref{lem:Wn}, we conclude that
neither $u_3$ nor $v^*$ is stressed at the first stage.
Thus ${\sf sd}^*(G_5)\geq 2$.

\begin{figure}
\centering
\begin{minipage}{0.24\textwidth}
\centering
\includegraphics[scale=0.5]{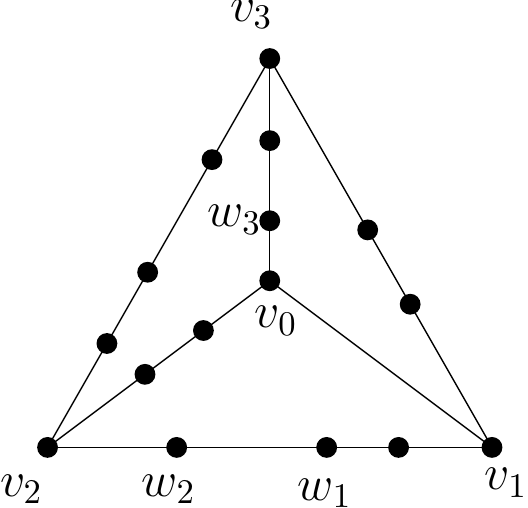}
\par
(a) $G_4$
\end{minipage}
\begin{minipage}{0.24\textwidth}
\centering
\includegraphics[scale=0.5]{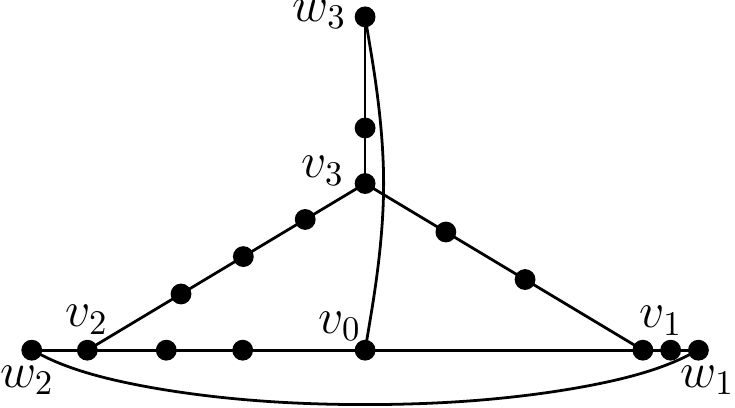}
\par
(b) $(G_4,p_4)$
\end{minipage}
\begin{minipage}{0.24\textwidth}
\centering
\includegraphics[scale=0.5]{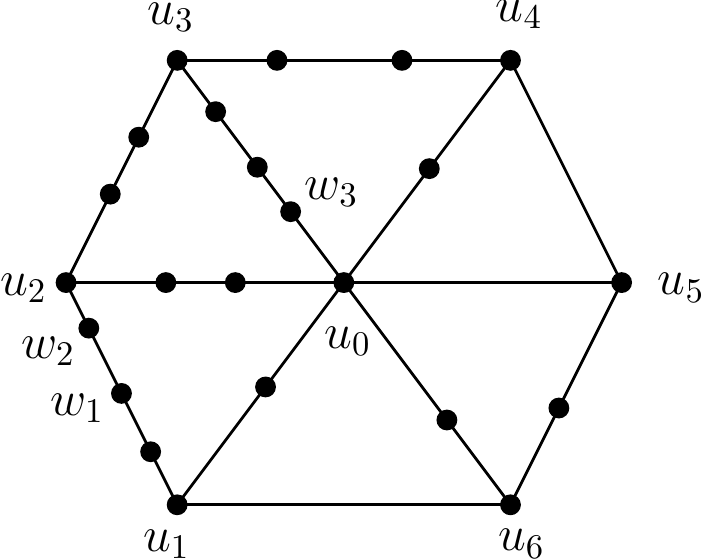}
\par
(c) $G_5$
\end{minipage}
\begin{minipage}{0.24\textwidth}
\centering
\includegraphics[scale=0.5]{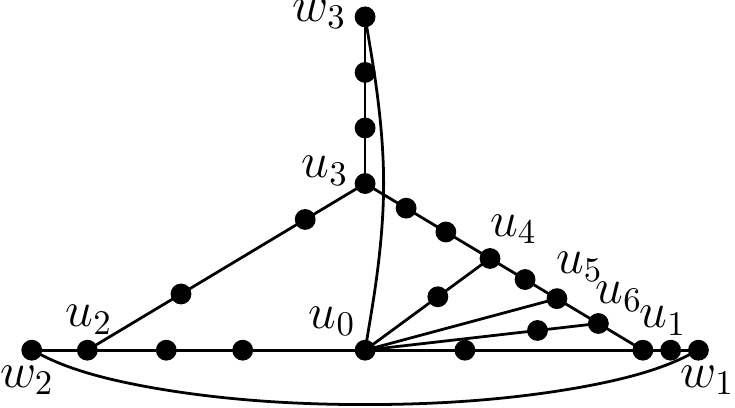}
\par
(d) $(G_5, p_5)$
\end{minipage}

\begin{minipage}{0.24\textwidth}
\centering
\includegraphics[scale=0.5]{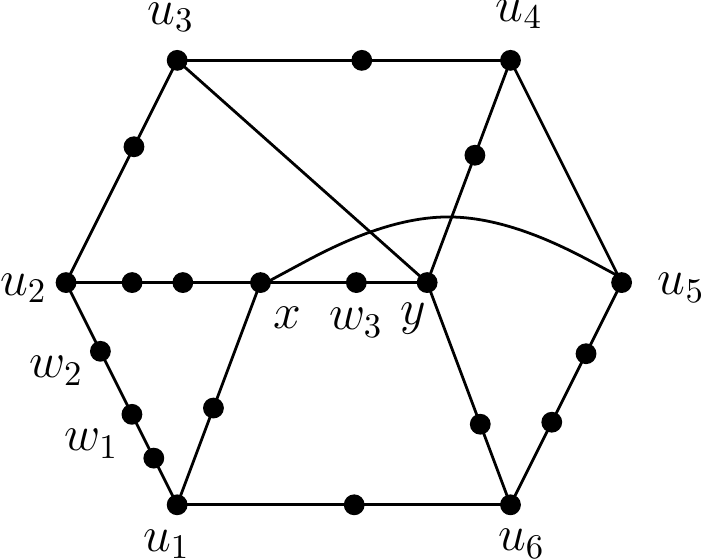}
\par
(e) $G_6$
\end{minipage}
\begin{minipage}{0.24\textwidth}
\centering
\includegraphics[scale=0.5]{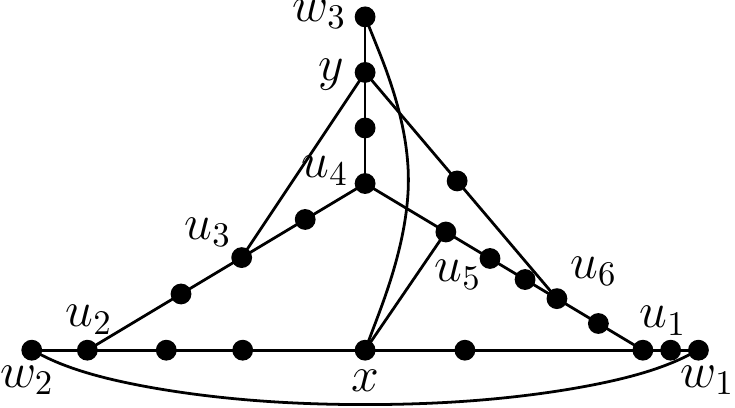}
\par
(f) $(G_6, p_6)$
\end{minipage}
\begin{minipage}{0.24\textwidth}
\centering
\includegraphics[scale=0.5]{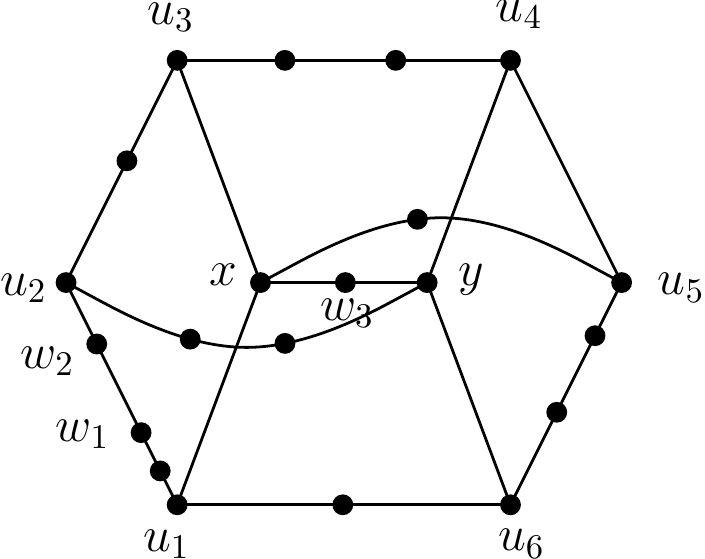}
\par
(g) $G_7$
\end{minipage}
\begin{minipage}{0.24\textwidth}
\centering
\includegraphics[scale=0.5]{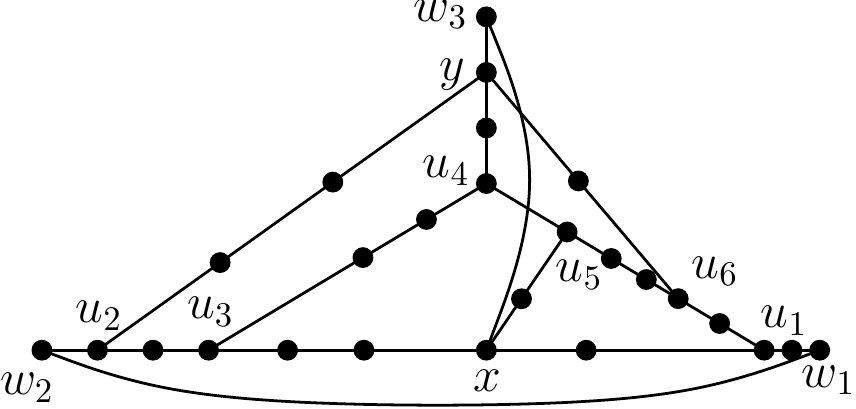}
\par
(h) $(G_7, p_7)$
\end{minipage}
\caption{(a)  A subdivision $G_4$ of $K_4$, (c) a subdivision $G_5$ of $W_7$, and (e)(g) subdivisions $G_6$ and $G_7$ of splittings of $W_7$. (b)(d)(f)(h) Projections of $(G_4, p_4)$, $(G_5, p_5)$, $(G_6, p_6)$, and $(G_7, p_7)$ to the plane, respectively.}
\label{fig:G4}
\end{figure}

\medskip
\noindent
Case 2: Suppose that $G$ is a subdivision of a graph obtained from $W_n$ by splitting the center vertex.
Let $x$ and $y$ be the vertices of $G$ obtained as a result of splitting the center vertex $u_0$.
Also let $G'$ be the graph obtained from $G$ by smoothing all the degree-two vertices
(that is, replacing each degree-two vertex $v$ and the edges incident to $v$ with an edge between the neighbors of $v$). We have the following two subcases.

\noindent
(case 2-1) Suppose that $G'$ has a triangle, say $xu_1u_2$ as in Figure~\ref{fig:G4}(e). 
Denote $G$ by $G_6$ in this case.
Take any path from $x$ to some $u_j$ with $j\neq 1,2$ so that it passes through $y$.
(In Figure~\ref{fig:G4}(e), we take $j=4$.) 
By identifying $u_i$ with $v_i$ for $i=1,2$, $u_j$ with $v_3$, and $x$ with $v_0$,  
$G_6$ contains $G_4$ as a subgraph, and spherical framework $(G_4, p_4)$ can be extended to $(G_6, p_6)$ by realizing the remaining paths by straight lines shown in Figure~\ref{fig:G4}(f).
Then ${\rm P}(G_6, p_6)$ has a unique solution by Lemma~\ref{lem:G4},
and at least $y$ or $v_j$ must be stressed at the final stage.
However,  since every vertex except $x$ and $y$ have degree at most three, 
the equilibrium condition implies that neither $y$ nor $u_j$ is stressed at the first stage, 
implying that ${\sf sd}^*(G_6)\geq 2$.

\noindent
(case 2-2) Suppose that $G'$ has no triangle as in Figure~\ref{fig:G4}(g). 
In this case $n\geq 5$.
Denote $G$ by $G_7$ in this case.
We may assume that $xu_1u_2u_3$ forms an induced cycle of length four.
Since $n\geq 5$, we can take a path from $x$ to some $u_j$ with $j\neq 1,2, 3$ so that it passes through $y$.
By identifying $u_1$ with $v_1$, $u_3$ with $v_2$,  
$u_j$ with $v_3$, and $x$ with $v_0$, $G_7$ contains $G_4$, and spherical framework $(G_4, p_4)$ can be extended to $(G_7, p_7)$ by realizing the remaining paths by straight lines.
Then ${\rm P}(G_7, p_7)$ has a unique solution by Lemma~\ref{lem:G4}.
Hence by the same reason as that for $(G_6, p_6)$  
$y$ or $u_j$ must be stressed at the final stage but neither $y$ nor $u_j$ is stressed at the first stage, 
implying that ${\sf sd}^*(G_7)\geq 2$.

This completes the proof of Lemma~\ref{lem:nonsingular} as well as that of Theorem~\ref{thm:nonsingular}. 
\end{proof}

\section{Bounding the singularity degree by two}
\label{subsec:two}

In this section we identify a large class of graphs whose singularity degree is bounded by two.
The following Lemma~\ref{lem:contraction2} establishes a connection  between ${\sf sd}(G,\Sigma)$ and ${\sf sd}^*(G,\Sigma)$.
For the proof we need the following two technical lemmas.

\begin{lemma}
\label{lem:contraction}
Let $(G,\Sigma)$  be a signed graph and $c\in {\cal E}(G,\Sigma)$.
Let $F=\{e\in E(G)\setminus \Sigma \mid c(e)=1\}$, $(G/ F, \Sigma)$ be the signed graph from $(G,\Sigma)$ obtained by contracting $F$, and $\tilde{c}$ be the restriction of $c$ to $E(G/F)$.
Then $\tilde{c}\in {\cal E}(G/F, \Sigma)$ and ${\sf sd}(G,\Sigma, c)\leq {\sf sd}(G/F,\Sigma,\tilde{c})+1$.
\end{lemma}
\begin{proof}
Let $e=ij\in F$. 
Since $c(e)=1$, $X[i,j]=1$ and $X[i,k]=X[j,k]$ for any feasible $X$ and any $k\in V\setminus \{i,j\}$.
This implies that 
there is a one-to-one correspondence between the feasible region of ${\rm P}(G, \Sigma, c)$ and that of 
${\rm P}(G/e, \Sigma, \tilde{c})$.
Thus $\tilde{c}\in {\cal E}(G/F, \Sigma)$.

To see the second claim, take a sequence $\tilde{\omega}^1, \dots, \tilde{\omega}^k$ of stresses of $G/F$ obtained by the facial reduction to P$(G/F,\Sigma, \tilde{c})$.

Consider the graph $(V,F)$, and let  $X_1, \dots, X_t$ be the family of the vertex sets of the connected components of $(V,F)$.
We denote the vertex  corresponding to $X_i$ in $G/F$ by $v_i$.
Then for each $1\leq i\leq t$ we define a sequence $\omega^1, \dots, \omega^k$ of stresses of $G$ by 
\begin{align*}
\omega^i(v)&=\tilde{\omega}^i(v_j)/|X_j| \quad \text{ if $v\in X_j$} \\
\omega^i(e)&=\begin{cases}
0 & \text{ if $e$ is induced by some $X_j$} \\ 
\tilde{\omega}_i(e) & \text{ otherwise}
 \end{cases}
 \end{align*}
 We also define a stress $\omega^0$ of $G$ such that
 \[\Omega^0=\sum_{uv \in F }({\bm e}_{u}-{\bm e}_v)({\bm e}_u-{\bm e}_v)^{\top}.
 \]
 We show that the sequence $\omega^0, \omega^1, \dots, \omega^k$ 
 satisfies (c1)(c2)(c3)(c4) for ${\rm P}(G,\Sigma,c)$.
 
 To see this we first show that $\omega^0$ is properly signed PSD equilibrium stress.
 By definition, $\Omega^0\succeq 0$.
 Also it is properly signed since $\omega^0(e)<0$ only if $e\in F$.
 It satisfies (c3) because 
 $\sum_{v\in V(G)} \omega^0(v)=2|F|=-\sum_{e\in F} \omega^0(e) c(e)$ by $c(e)=1$ for $e\in F$.
 
 Next we examine the remaining of the sequence $\omega^1, \dots, \omega^k$. 
 Note that 
 for any $x\in \mathbb{R}^{V}$ we have  
 \begin{align*}
 x^{\top} \Omega^0 x=0   \Leftrightarrow  \Omega^0 x=0
 \Leftrightarrow  x\in \bigcap_{uv\in F}{\rm ker}\ ({\bm e}_{u}-{\bm e}_v)({\bm e}_u-{\bm e}_v)^{\top}  \Leftrightarrow x(u)=x(v) \text{ for every $uv\in F$}. 
 \end{align*}
 In other words, 
 $\{x\in \mathbb{R}^V\mid x^{\top} \Omega^0 x=0\}$ is equal to 
 \[
 \{x\in \mathbb{R}^V\mid 
 \text{$x(u)=x(v)$ if $u$ and $v$ belong to the  same component in graph $(V,F)$}\}.
 \]
 This linear space is naturally identified with $\mathbb{R}^{V(G/F)}$,
 and the fact that $\omega^0,\dots, \omega^k$ satisfies (c1)(c2)(c3)(c4) follows from 
 the fact that $\tilde{\omega}^1,\dots, \tilde{\omega}^k$ satisfies (c1)(c2)(c3)(c4).
  \end{proof}
 
 Recall resigning operations for signed graphs defined in Subsection~\ref{subsec:K4}.
 We define the corresponding operation for edge vectors as follows.
 Let $(G,\Sigma)$ be a signed graph and  $S\subseteq V(G)$.
For $c\in [-1,1]^{E}$,   the {\em resigning} $c^S\in [0,1]^{E}$ of $c$ with respect to $S\subseteq V$ is defined by 
$c^S(e)=-c(e)$ if $e \in \delta(S)$ and
$c^S(e)=c(e)$ if $e\notin \delta(S)$.
Observe that $c\in {\cal E}(G,\Sigma)$ if and only if $c^S\in {\cal E}(G,\Sigma\Delta \delta(S))$.

 \begin{lemma}
\label{lem:switching}
Let $(G,\Sigma)$ be a signed graph, $c\in [-1,1]^{E(G)}$, and $S\subseteq V(G)$.
Then ${\sf sd}(G,\Sigma, c)={\sf sd}(G, \Sigma \Delta \delta(S), c^S)$.
\end{lemma}
\begin{proof}
Let $X$ be a maximum rank solution of ${\rm P}(G, \Sigma, c)$ and $\omega^1,\dots, \omega^k$ be the sequence obtained by the facial reduction to ${\rm P}(G,\Sigma, c)$.
Let $D$ be a diagonal matrix of size $|V(G)|$ whose $i$-the diagonal entry is $-1$ if $i\in S$ and $1$ if $i\in V(G)\setminus S$.
Observe that $DXD$ is feasible in ${\rm P}(G, \Sigma \Delta \delta(S), c^S)$.
Also $D\Omega^iD$ is the stress matrix of a stress $\tilde{\omega}^i$ supported on $E(G)$.
We claim that $\tilde{\omega}^1,\dots, \tilde{\omega}^k$ form a sequence of stresses satisfying (c1)(c2)(c3)(c4).
Indeed, $\tilde{\omega}^i(uv)=-\tilde{\omega}^i(uv)$ if $uv\in \delta(S)$ and otherwise
 $\tilde{\omega}^i(uv)=\tilde{\omega}^i(uv)$. Hence (c1) holds.
 (c2) trivially holds since $\tilde{\Omega}^i=D\Omega^iD$ and $\Omega^i$ satisfies (c2).
 To check (c3), one can directly check (\ref{eq:semi-equilibrium}) for any $p^S$ with ${\rm Gram}(p^S)\in  {\rm P}(G, \Sigma \Delta \delta(S), c^S)$.
  (c4) holds since each $\tilde{\Omega}^i$ is obtained from 
   $\Omega^i$ by the same coordinate change for all $i$.
   
  Applying the same argument starting from ${\rm P}(G, \Sigma\Delta \delta(S), c^S)$, we also have the opposite direction.
  Thus we get ${\sf sd}(G,\Sigma, c)={\sf sd}(G, \Sigma \Delta \delta(S), c^S)$.
\end{proof}

 We say that $(G, \Sigma)$ is a {\em splitting} of $(G',\Sigma')$ if $(G',\Sigma')$ can be obtained from $(G,\Sigma)$ by a sequence of resigning and contraction of even edges.
 Combining Lemma~\ref{lem:contraction} with resigning operations we have the following.
 \begin{lemma}
 \label{lem:contraction2}
Let $(G,\Sigma)$  be a signed graph.
Then we have
\begin{equation}
\label{eq:contraction2}
{\sf sd}(G, \Sigma)\leq 
\max\left\{{\sf sd}^*(H,\Sigma_H)+1 : 
\text{ $(G,\Sigma)$ is a splitting of $(H, \Sigma_H)$}  
\right\}
\end{equation}
\end{lemma}
\begin{proof}
Take any $c\in {\cal E}(G,\Sigma)$ with ${\sf sd}(G,\Sigma)={\sf sd}(G,\Sigma,c)$.
Then by Theorem~\ref{thm:signed} we have ${\rm arccos}(c)/\pi\in {\rm MET}(G,\Sigma)$.
Let $F^+=\{e\in E\setminus \Sigma: c(e)=1\}$ and $F^-=\{e\in \Sigma: c(e)=-1\}$.
Then observe that $F^-$ (edge-)induces a bipartite graph in $G/F^+$ since otherwise $G$ has  an odd  cycle $C$ consisting of edges in $F^-\cup F^+$, 
and as each edge in $F^-\cup F^+$ is degenerate we would have $\sum_{e\in C}\sigma(e){\rm arccos}( c(e))=-|E(C)\cap \Sigma|\pi$, contradicting ${\rm arccos}(c)/\pi\in {\rm MET}(G,\Sigma)$.
Therefore there is a cut $\delta(S)$ in $G$ such that $F^-\subseteq \delta(S)$ and $F^+\cap \delta(S)=\emptyset$.
Let $(G, \Sigma \Delta \delta(S))$ and $c^S$ be the resigning of $G$ and $c$ with respect to $\delta(S)$, respectively.
Then every edge $e\in \Sigma$ satisfies $c(e)>-1$.

Let $F=F^+\cup F^-$.
Also let $(H, \Sigma_H)=(G/F, \Sigma \Delta \delta(S))$ and  $c_H$ be the restriction of $c^S$ to $E(H)$.
Then $(G, \Sigma)$ is a splitting of $(H, \Sigma_H)$ and $c_H$ is nondegenerate.
By Lemma~\ref{lem:switching} and Lemma~\ref{lem:contraction}, 
we obtain ${\sf sd}(G,\Sigma, c)={\sf sd}(G, \Sigma \Delta \delta(S), c^S)\leq 
{\sf sd}(H, \Sigma_H, c_H)+1$, which is at most the right hand side of (\ref{eq:contraction2}).
\end{proof}

By Corollary~\ref{cor:K4} and Lemma~\ref{lem:contraction2} we obtain the upper bound of ${\sf sd}(G, \Sigma)$ for odd-$K_4$ minor free signed graphs.

\begin{corollary}
\label{cor:K4_2}
Let $(G,\Sigma)$ be an odd-$K_4$ minor free signed graph. 
Then ${\sf sd}(G,\Sigma)\leq 2$.
\end{corollary}

Combining Lemma~\ref{lem:clique_sum}, Corollary~\ref{cor:K4_2}, and (ii)$\Leftrightarrow$(iii) in Theorem~\ref{thm:nonsingular}, we also have the following.
\begin{corollary}
\label{thm:singular_two}
If $G$ has neither $W_n\ (n\geq 5)$ nor a splitting of $W_n\ (n\geq 4)$ as an induced subgraph,
then ${\sf sd}(G)\leq 2$.
\end{corollary}
Corollary~\ref{thm:singular_two} is not tight since a graph obtained from $K_4$ by subdividing an edge once has singularity degree equal to two.
Characterizing graphs with singularity degree equal two is left as an open problem.

We also remark that Theorem~\ref{thm:singular_two}  holds even if ${\sf sd}(G)$ is defined over $c:E(G) \rightarrow \mathbb{R}$ (and $c$ is not necessarily to be in ${\cal E}(G)$). 
If $c\notin {\cal E}(G)$ then there is a violation to the metric inequality or the clique inequality (i.e., the PSD condition on cliques) due to the description of ${\cal E}(G)$ by Barret, Johnson, and Loewy~\cite{bjl}. This violation can be detected by the first stage of the facial reduction if $c$ is nondegenerate. 

\section{Example of Large Singularity Degree}
\label{sec:large}
In this section we give a family of graphs whose singularity degree grows linearly  in the number of vertices.
The example $(G^k, q^k)$ consists of a graph $G^k$ and $q^k:V(G^k)\rightarrow \mathbb{
S}^{k-1}$,  which are defined inductively from $k=1$ as follows:
\begin{itemize}
\setlength{\parskip}{0.cm} 
 \setlength{\itemsep}{0.05cm} 
\item  $V(G^1)=\{v_1(=w_1)\}$ and $E(G^1)=\emptyset$;
\item $V(G^k)=V(G^{k-1})\cup \{u_k, v_k, w_k\}$;
\item $E(G^k)=E(G^{k-1})\cup \{w_{k-1}v_k, w_{k-1}w_k, u_kv_{k-1}, u_kv_k, u_kw_k\}$;
\item $q^k(v_i)=q^k(w_i)={\bm e}_i$; 
\item $q^k(u_i)$ is placed in the span of $\{{\bm e}_i, {\bm e}_{i-1}\}$ such that 
$q^k(v_i)$ is on the spherical line between $q^k(v_{i-1})$ and $q^k(u_i)$ for $2\leq i\leq k$.
\end{itemize}
See Figure~\ref{fig:examples}.

\begin{figure}
\centering
\begin{minipage}{0.48\textwidth}
\centering
\includegraphics[scale=0.45]{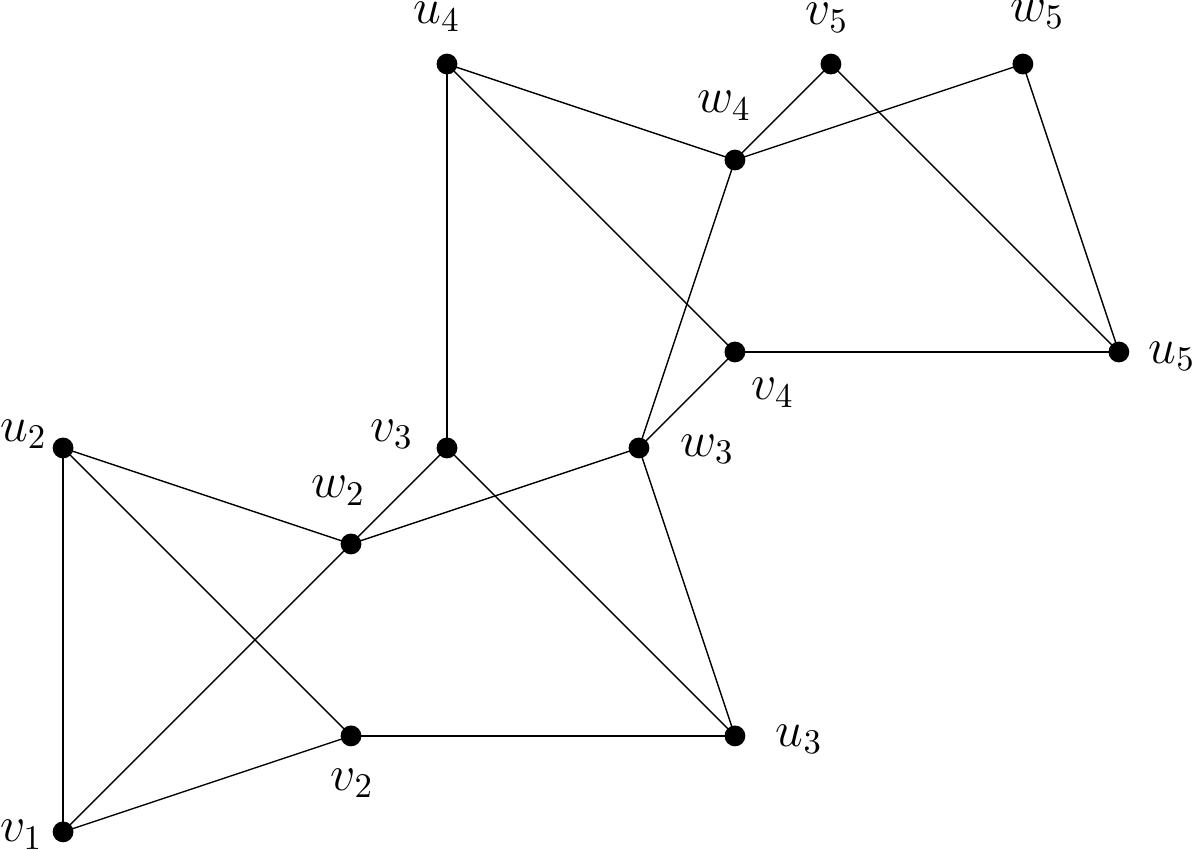}
\par (a)
\end{minipage}
\begin{minipage}{0.46\textwidth}
\centering
\includegraphics[scale=0.53]{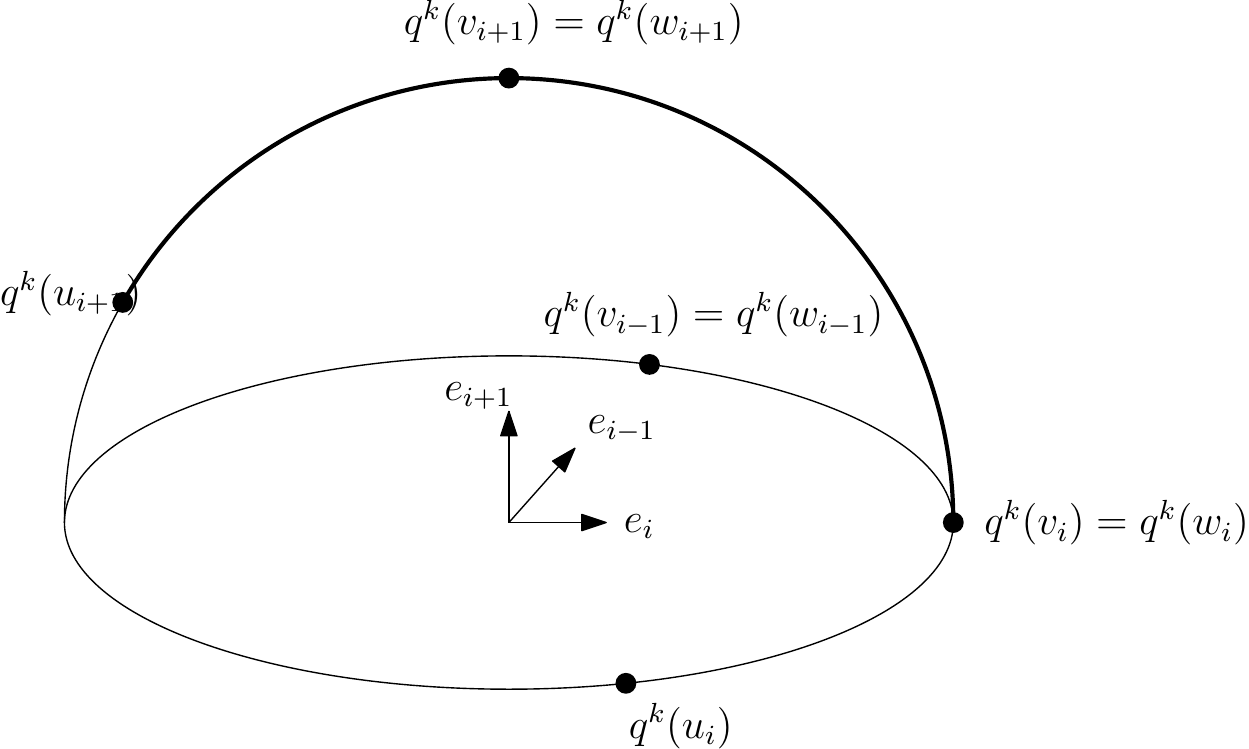}
\par (b)
\end{minipage}
\caption{(a) $G^5$ and (b) a part of $(G^k, q^k)$.}
\label{fig:examples}
\end{figure}

\begin{theorem}
\label{thm:example2}
Let $(G^k,q^k)$ be a spherical framework as defined above, let $X^k={\rm Gram}(q^k)$, 
and let $c^k=\pi_{G^k}(X^k)$.
Then ${\sf sd}(G^k, c^k)\geq k-1=(|V(G^k)|-1)/3$.
\end{theorem}
\begin{proof}
Due to the definition of $q^k$, for any feasible ${\rm Gram}(p)$  of ${\rm P}(G^k, c^k)$, we have $p(v_i)=p(w_i)$ for every $i$ with $1\leq i\leq k$. 
In particular $p(v_k)=p(w_k)$, and hence $v_k$ and $w_k$ are stressed at the last stage of the facial reduction to ${\rm P}(G^k, c^k)$.
%

Let $\omega^1,\dots, \omega^h$ be the sequence of the stresses obtained by the facial reduction,
and let $V_j$ be the set of stressed vertices at the $j$-th stage.
The statement follows by showing that 
\begin{equation}
\label{eq:ex0}
\{u_i, v_i, w_i\}\cap V_{j}=\emptyset \text{  for every $i, j$ with $i\geq j+2$.}
\end{equation}

The proof is done by induction on $j$.
The claim is trivial for $j=0$, and hence assume it for the $(j-1)$-th stage.
Then we have $q^k(V_{j-1})\subseteq {\rm span}\{{\bm e}_1,\dots, {\bm e}_{j}\}$ by induction and by the definition of $q^k$. 
In particular, ${\bm e}_i$ is orthogonal to $\spa q^k(V_{j-1})$ for $i\geq j+1$.

Consider the equilibrium condition at $v_{i-1}$ projected to the span of ${\bm e}_i$ for $i\geq j+2$.
Then Lemma~\ref{lem:singularity_stress} implies that
\begin{align*}
0=\left( \omega^j(v_{i-1}) p(v_{i-1})+\omega^j(u_{i-1}v_{i-1}) p(u_{i-1})+\omega^j(w_{i-2}v_{i-1}) p(w_{i-2})+\omega^j(u_iv_{i-1}) p(u_i)\right)\cdot {\bm e}_i. 
\end{align*}
Since  $p(v_{i-1})\cdot {\bm e}_i=p(u_{i-1})\cdot {\bm e}_i=p(w_{i-2})\cdot {\bm e}_i=0$ while   $p(u_i)\cdot {\bm e}_i \neq 0$, we get
\begin{equation}
\label{eq:ex1}
\omega^j(u_iv_{i-1})=0 \qquad (i\geq j+2).
\end{equation}

Next, consider the equilibrium condition at $u_i$ projected to the span of ${\bm e}_{i-1}$ for $i\geq j+2$.
By (\ref{eq:ex1})  we have 
\begin{align*}
0=\left( \omega^j(u_{i}) p(u_{i})+\omega^j(u_{i}v_{i}) p(v_{i})+\omega^j(u_iw_{i}) p(w_{i})\right)\cdot {\bm e}_{i-1}. 
\end{align*}
Since  $p(v_{i})\cdot {\bm e}_{i-1}=p(w_{i})\cdot {\bm e}_{i-1}=0$ while   
$p(u_i)\cdot {\bm e}_{i-1} \neq 0$, we get
\begin{equation}
\label{eq:ex2}
\omega^j(u_i)=0 \qquad (i\geq j+2).
\end{equation}

Since $u_i, v_i, w_i$ are not stressed at the $(j-1)$-the stage, 
$\Omega^j[\{u_i, v_i, w_i\}, \{u_i, v_i, w_i\}]\succeq 0$ by Lemma~\ref{lem:singularity_stress}.
However, by (\ref{eq:ex2}), the first diagonal of $\Omega^j[\{u_i, v_i, w_i\}, \{u_i, v_i, w_i\}]$ is equal to zero, which means that 
\begin{equation}
\label{eq:ex3}
\omega^j(u_iv_i)=\omega^j(u_iw_i)=0 \qquad (i\geq j+2).
\end{equation}

By (\ref{eq:ex3}),  the equilibrium conditions at $v_i$ and at $w_i$ 
projected to the span of ${\bm e}_{i-1}$ give
\begin{equation}
\label{eq:ex4}
\omega^j(w_{i-1}v_i)=\omega^j(w_{i-1}w_i)=0\qquad (i\geq j+2)
\end{equation}
while the equilibrium conditions at $v_i$ and at $w_i$ 
projected to the span of ${\bm e}_{i}$ give 
\begin{equation}
\label{eq:ex5}
\omega^j(v_i)=\omega^j(w_i)=0\qquad (i\geq j+2).
\end{equation}
In total the stress of each edge incident to $u_i, v_i$ or $w_i\ (i\geq j+2)$ is equal to zero, and we get (\ref{eq:ex0}).
Since $v_k$ and $w_k$ are stressed at the last stage, we must have $k<h+2$, meaning $h\geq k-1$.
\end{proof}
Notice that $G^k$ has treewidth equal to three.
Combining this result with Lemma~\ref{lem:induced_subgraph} we have the following.
\begin{corollary}
\label{cor:large}
For each positive integer $n$, there is a graph $G$ of $n$ vertices whose treewidth is equal to three
and ${\rm sd}(G)\geq \lfloor (n-1)/3\rfloor$.
\end{corollary}

Corollary~\ref{cor:large} gives a negative answer to a question posed by Anthony Man-Cho So at BIRS Global Rigidity workshop in July 2015 in Banff, who asked whether the singularity degree of the  matrix completion problem can be bounded by a sublinear function in the number of vertices. (His question was posed in terms of the Euclidean matrix completion problem. For the Euclidean matrix completion the projection of $(G^k,q^k)$ to the plane gives an example.) 

\section{Super Stability}
\label{sec:super}
Following a terminology introduced by Connelly~\cite{c}, we say that 
a spherical framework $(G, p)$ in $\mathbb{S}^{d-1}$ is {\em super stable} 
if there exists an equilibrium PSD stress $\omega$ such that
$\corank \Omega=d$ and 
\begin{equation}
\label{eq:conic}
\text{there is no nonzero }S\in {\cal S}^d  \text{ satisfying } p(i)^{\top} S p(j)=0\ \text{for all } ij\in V(G)\cup J_{\omega},
\end{equation}
where $J_{\omega}=\{ij\in E(G)\mid \omega(ij)\neq 0\}$.
Note that a certificate $\omega$ for super stability  coincides with a length-one  universal rigidity certificate given in Proposition~\ref{prop:universal}.
Therefore,  super stability is a sufficient condition for  universal rigidity.

One of recent major topics in rigidity is to understand the exact relation between super stability and universal rigidity~\cite{gt,cg15,cg,jh,cg15b}. 
It was shown by Gortler and Thurston~\cite{gt}, for the unsigned case, that
universal rigidity is equivalent to super stability if $p$ is generic (i.e., the set of coordinates is algebraically independent modulo the defining equation of the unit sphere).
Recently, Connelly and Gortler~\cite{cg15b} gave an algebraic characterization of super stable frameworks. 
Our analysis of the singularity degree can be adapted to establish the following new relation between super stability and universal rigidity at the level of graphs.
\begin{theorem}
\label{thm:super}
Let $G$ be a graph.
Any universally rigid spherical framework $(G,p)$  of $G$ is super stable 
if and only if $G$ is chordal.
\end{theorem}
\begin{proof}
We only give a sketch of the poof.
The sufficiency follows by observing that the equivalence between universal rigidity and super stability is preserved by clique sum. This can be checked by showing that (\ref{eq:conic}) is preserved in the construction of $\Omega$ in the proof of Lemma~\ref{lem:clique_sum}.

The necessity follows by observing that the example $(C_n,p)$ given in the proof of Lemma~\ref{lem:cycle} is universally rigid, and the universal rigidity is preserved in the construction of frameworks in the proof of Lemma~\ref{lem:induced_subgraph}.
\end{proof}

We say that a spherical framework $(G,p)$ is {\em nondegenerate} if $p(i)\cdot p(j)\neq \pm 1$ for every edge $ij\in E$.
\begin{theorem}
Let $G$ be a graph.
Any nondegenerate universally rigid spherical framework $(G,p)$ of $G$ is super stable
if and only if $G$ has neither $W_n\ (n\geq 5)$ nor a proper splitting of $W_n\ (n\geq 4)$ as an induced subgraph, 
or equivalently, $G$ is the clique sum of chordal graphs and $K_4$-minor free graphs.
\end{theorem}  
\begin{proof}
In \cite[Corollary 6.3]{t} it was shown that, if $(G,\Sigma)$ is odd-$K_4$-minor free and $(G,\Sigma, p)$ is nondegenerate,
then  $(G, \Sigma, p)$ is universally rigid if and only if $(G, \Sigma, p)$ is super stable.
This in particular implies that  universal rigidity is equivalent to super stability for every nondegenerate spherical framework $(G, p)$
if $G$ is $K_4$-minor free.
Since the property is preserved by clique sum, 
the sufficiency follows from (ii)$\Leftrightarrow$(iii) in Theorem~\ref{thm:nonsingular}.

The necessity follows by observing that the frameworks given in the proof of Lemma~\ref{lem:nonsingular} are universally rigid, and the universal rigidity is preserved in the construction of frameworks in the proof of Lemma~\ref{lem:induced_subgraph}.
\end{proof}

\section{Conclusion}
We have introduced new graph parameters, singularity degree and nondegenerate singularity degree, and gave a characterization of the class of graphs whose parameter value is at most one for each parameter.

There are quite a few  remaining questions. 
An obvious question is to give a characterization of graphs with singularity degree at most two.
Another interesting open problem is to give a characterization of signed graphs whose singularity degree is at most one. Since Lemma~\ref{lem:cycle} can be extended to any odd cycle of length $n\geq 4$ in the signed setting, Lemma~\ref{lem:induced_subgraph} implies that, if ${\sf sd}(G,\Sigma)\leq 1$, then $(G,\Sigma)$ has no odd hole.
However the structure of odd-hole free signed graphs seems to be much more complicated than that of chordal graphs (see \cite{ccv} for a special case).
Extending Theorem~\ref{thm:super} to signed graphs also remains unsolved.
%

Using  a linear correspondence between an open hemisphere and an Euclidean space, one can translate each argument of this paper  to establish the corresponding result for the Euclidean matrix completion problem. However, an exact relation of the singularity degrees of the two completion problems is not clear.

\section*{Acknowledgement}
The author would like to thank  Monique Laurent for fruitful discussions on the topic of this paper.
An example of large singularity degree for a cycle was first pointed out by her. 
She also suggested several key notions used in this paper to analyze the singularity degree.

The author would like to thank Henry Wolkowicz for bringing our attention to the case when the singularity degree is equal to  zero.

This work was supported by JSPS Postdoctoral Fellowships for Research Abroad, 
JSPS Grant-in-Aid for Young Scientist (B) 15K15942, 
and JSPS Grant-in-Aid for Scientific Research (C) 15KT0109.


\begin{thebibliography}{9}

\bibitem{ahm}
F.~Alizadeh,  J.~P.~A. Haeberly, and M.~L.~Overton. 
Complementarity and nondegeneracy in semidefinite programming. 
Math.~Prog., 77, 111--128, 1997.


\bibitem{a15}
A.~Y.~Alfakih. 
On Farkas lemma and dimensional rigidity of bar frameworks.
Linear Algebra Appl., 486, 504--522, 2015.

%

\bibitem{bjl}
W.~W.~Barrett, C.~R.~Johnson, and R.~Loewy. 
The real positive definite completion problem: cycle completability. 
Mem.~Amer.~Math.~Soc.,
584, 69 pages, 1996.

\bibitem{bw}
M.~J.~Borwein and H.~Wolkowicz. 
Facial reduction for a cone-convex programming problem. 
J.~Aust.~Math.~Soc., 30, 369--380, 1981.


\bibitem{ccv}
M.~Conforti, G.~Cornu{\'e}jols, and K.~Vu{\u s}kovi{\'c}.
Decomposition of odd-hole-free graphs by double star cutsets and 2-joins.
Discrete Appl.~Math., 141:41--91, 2004.

\bibitem{c}
R.~Connelly. Rigidity and energy. 
Invent.~Math., 66, 11--33, 1982.

\bibitem{c11}
R.~Connelly.  Combining globally rigid frameworks.
Proc.~Steklov Inst.~Math., 274:191--198, 2011


\bibitem{cg}
R.~Connelly and S.~Gortler. 
Iterative universal rigidity. 
Discrete Comput.~Geom., 53(4), 847--877, 2015.

\bibitem{cg15}
R.~Connelly and S.~Gortler. 
Universal rigidity of complete bipartite.
arXiv:1502.02278v1, 2015.

\bibitem{cg15b}
R.~Connelly and S.~Gortler. 
Prestress stability of triangulated convex polytopes and universal second order rigidity.
arXiv:1510.04185v1, 2015.




\bibitem{dpw}
D.~Drusvyatskiy, G.~Pataki, and H.~Wolkowicz.
Coordinate Shadows of Semidefinite and Euclidean Distance Matrices.
SIAM J.~Optim., 25, 1160--1178, 2015.


%
\bibitem{gt}
S.~Gortler and D.~Thurston. 
Characterizing the universal rigidity of generic frameworks, 
Discrete Comput.~Geom., 51, 1017--1036, 2014.

%




\bibitem{jm}
C.~R.~Johnson and T.~A.~McKee. 
Structural conditions for cycle completable graphs. Discrete Math., 159, 155--160, 1996. 

\bibitem{jh}
T.~Jord{\'a}n and V.~H.~Nguyen.
On universally rigid frameworks on the line.
 Contributions to Discrete Mathematics, 10, 10--21, 2015.


\bibitem{l97}
M.~Laurent. 
The real positive semidefinite completion problem for series-parallel graphs. 
Linear Algebra Appl.,  252, 347--366, 1997.

\bibitem{l98}
M.~Laurent. 
A tour d'horizon on positive semidefinite and Euclidean distance matrix completion problems. 
In P.~Pardalos and H.~Wolkowicz, editors, Topics in Semidefinite and Interior-Point Methods, volume 18 of The Fields Institute for Research in Mathematical Science, Communications Series. Providence, Rhode Island, pages 51--76, 1998. 


\bibitem{lv}
M.~Laurent and A.~Varvitsiotis. 
Positive semidefinite matrix completion, universal rigidity and the strong Arnold property. 
Linear Algebra Appl., 452, 292--317, 2014.

\bibitem{lp}
M.~Liu and G.~Pataki. 
Exact duality in semidefinite programming based on elementary reformulations. 
SIAM J.~Optim., 25, 1441--1454, 2015.



\bibitem{s00}
J.~F.~Sturm.
Error bounds for linear matrix inequalities.
SIAM J. Optim., 10(4), 1228--1248, 2000.

\bibitem{t}
S.~Tanigawa.
The signed positive semidefinite matrix completion problem for odd-$K_4$ minor free signed graphs.
 arXiv:1603.08370, 2016.
 

\bibitem{zsy}
Z.~Zhu, A.~Man-Cho So, and Y.~Ye. Universal rigidity and edge sparsification for sensor network localization.
SIAM J.~Optim. 20, 3059--3081, 2010.



\end{thebibliography}
\end{document}